\documentclass[10pt]{amsart}

\usepackage{amsmath}
\usepackage{amssymb}
\usepackage{amsthm}
\usepackage{textcomp}
\usepackage{indentfirst}
\usepackage{mathrsfs}
\usepackage{tikz-cd}
\usepackage{epsfig}
\usepackage{color}
\usepackage{bbm}
\usepackage[yyyymmdd,hhmmss]{datetime}
\usepackage{thmtools}
\usepackage{mathdots}
\declaretheoremstyle[headfont=\normalfont]{normalhead}

\usepackage[margin=1.175in]{geometry}
\usepackage{lipsum}

\usepackage[pagebackref]{hyperref}
\hypersetup{
    colorlinks = true,
    citecolor = [rgb]{0, 0.5, 0.00},
}

\renewcommand*{\backref}[1]{}
\renewcommand*{\backrefalt}[4]{%
    \ifcase #1 (Not cited.)%
    \or        (Cited on page~#2.)%
    \else      (Cited on pages~#2.)%
    \fi}

\newtheorem{thm}{Theorem}
\newtheorem{lem}[thm]{Lemma}
\newtheorem{prop}[thm]{Proposition}
\newtheorem{cor}[thm]{Corollary}
\newtheorem{defn}[thm]{Definition}
\newtheorem{rmk}[thm]{Remark}

\newtheorem{example}[thm]{Example}

\newcommand\G{\mathbf{G}}
\renewcommand\H{\mathbf{H}}
\newcommand\Q{\mathbf{Q}}
\renewcommand\L{\mathbf{L}}
\newcommand\V{\mathbf{V}}
\renewcommand\P{\mathbf{P}}
\newcommand\M{\mathbf{M}}
\newcommand\U{\mathbf{U}}
\newcommand\B{\mathbf{B}}
\newcommand\T{\mathbf{T}}
\newcommand\N{\mathbf{N}}

\newcommand\setc{\mathfrak{c}}
\newcommand\Z{\mathbb{Z}}
\newcommand\R{\mathbb{R}}
\newcommand\C{\mathbb{C}}
\newcommand\A{\mathbb{A}}
\newcommand\X{\mathbf{X}}
\newcommand\Au{\mathcal{A}}

\newcommand\GL{\mathbf{GL}}

\newcommand\Hom{\operatorname{Hom}}

\newcommand\Res{\operatorname{Res}}

\newcommand{\mo}{\backslash}
\newcommand{\Sp}{\mathbf{Sp}}

\newcommand{\e}{\`e}
\newcommand{\pair}[2]{\langle #1,#2 \rangle}
\newcommand{\Un}[1]{\mathbf{U_{#1}}}
\newcommand{\Spn}[1]{\mathbf{Sp_{#1}}}
\newcommand{\GLn}[1]{\mathbf{GL_{#1}}}
\newcommand{\ResGL}[1]{\mathbf{Res_{E/F}GL_{#1}}}
\newcommand{\ResSp}[1]{\mathbf{Res_{E/F}Sp_{#1}}}
\newcommand\diag{\operatorname{diag}}
\newcommand\pr{\operatorname{pr}}
\newcommand\MWM{{}_MW_M}
\newcommand\RR{\mathbf{R}}
\renewcommand\Z{\mathbf{Z}}

\newcommand{\frakA}{\mathfrak{A}}
\newcommand{\frakB}{\mathfrak{B}}
\newcommand{\frakC}{\mathfrak{C}}
\newcommand{\frakE}{\mathfrak{E}}
\newcommand{\frakF}{\mathfrak{F}}
\newcommand{\frakG}{\mathfrak{G}}
\newcommand{\frakP}{\mathfrak{P}}
\newcommand{\frakQ}{\mathfrak{Q}}
\newcommand{\frakS}{\mathfrak{S}}
\newcommand{\frakW}{\mathfrak{W}}
\newcommand{\frakX}{\mathfrak{X}}
\newcommand\fraka{\mathfrak{a}}
\newcommand{\frakc}{\mathfrak{c}}

\newcommand\picusp{\operatorname{\Pi_{cusp}}}       
\newcommand{\BG}{([\G])}                            
\newcommand{\UAMG}{\U(\A)M\backslash\G(\A)}         
\newcommand{\hdist}{H\text{-dist}}                  
\newcommand{\hconv}{H\text{-conv}}                  
\newcommand{\disc}{\text{disc}}                     
\newcommand{\cusp}{\text{cusp}}                     
\newcommand{\SHdist}{[S_\frakX]_{H\text{-dist}}}    
\newcommand{\PRFMpi}{P^{R,\frakF}_{(M,\pi)}}        
\newcommand{\PRFX}{P^{R,\frakF}_{\frakX}}        

\renewcommand\Re{\mathrm{Re}}

\title{On the $\mathrm{Sp}_{2n}$-distinguished automorphic spectrum of $\mathrm{U}_{2n}$}
\author{Kewen Wang, Yu Xin}
\begin{document}
\maketitle

\begin{abstract}
Given a reductive group $G$ and a reductive subgroup $H$, both defined over a
number field $F$, we introduce the notion of the $H$-distinguished automorphic spectrum of $G$ and analyze it for the pair $(\mathrm{U}_{2n},\mathrm{Sp}_{2n})$. We derived a formula for period integrals of pseudo-Eisenstein series of $\mathrm{U}_{2n}$ in analogy with the main result of Lapid and Offen in their work analyzing the pair $(\mathrm{Sp}_{4n},\mathrm{Sp}_{2n} \times \mathrm{Sp}_{2n})$. We give an upper bound of the distinguished spectrum with the formula. A non-trivial lower bound of the discrete distinguished spectrum is expected from the formula, given the previous work.
\end{abstract}
\tableofcontents

\section{Introduction}
Let $G$ be a reductive group over a number field $F$, and let $H$ be a closed subgroup of $G$ defined over $F$. Let $\A$ be a ring of adeles of $F$. In the theory of automorphic forms one is often interested in the period integral
\[
\int_{H(F) \mo H(\A)} \varphi(h) dh
\]
and in automorphic representations of $G(\A)$ on which such an integral is not identically zero. Such automorphic representations are called \textit{$H$-distinguished}.

The $H$-distinguished cuspidal spectrum is well defined as the linear functional of the period integral is always defined on a cuspidal representation (cf. \cite{AGR}). However, it is not always the case for general automorphic representations. Lapid and Offen proposed a candidate for this concept and studied it in the case when $(G,H) = (\mathrm{Sp}_{4n},\mathrm{Sp}_{2n}\times \mathrm{Sp}_{2n})$. The technical heart of their work is a formula of the period integral of pseudo-Eisenstein series based on the study of the double cosets $P \mo G/H$, where $P$ is a parabolic subgroup of $G$. Much of the work here can date further back to Jacquet, Lapid and Rogawski in \cite{JLR} and \cite{LR}.

In this paper, we followed the work \cite{LO} of Lapid and Offen, spiritually and technically, and studied the $H$-distinction for the pair $(G,H) = (\mathrm{U}_{2n},\mathrm{Sp}_{2n})$. We derived the formula for period integrals of pseudo-Eisenstein series of $\mathrm{U}_{2n}$. With the formula, we gave an upper bound of the $H$-distinguished spectrum in terms of the coarse decomposition
\[
L^2([G]) = \hat{\bigoplus}_\mathfrak{X} L^2_\mathfrak{X}([G])
\]
by excluding some cuspidal data. Moreover, for the remaining cuspidal data $\mathfrak{X}$ we will control the affine spaces $\mathfrak{S}$ which potentially contribute to the $H$-distinguished spectrum $L^2_{H-dist}([G])$ under the finer decomposition
\[
L^2_\mathfrak{X}([G]) = \bigoplus_\mathfrak{S} L^2_\mathfrak{X}([G])_\mathfrak{S}
\]
according to intersections of singular hyperplanes.

A natural extension of this work is to understand the discrete distinguished spectrum. As that in \cite{LO}, we expect to have a non-trivial discrete $H$-distinguished subrepresentation. However, we have not accomplished this goal.

\subsection*{Acknowledgement}
The majority of this work took place at Brandeis university where the authors studied as graduate students, while some complementary work took place at Bar-Ilan university where the authors worked as postdoctoral fellows. The authors would like to thank both the institutions for their support and excellent
working environments. In particular, the authors were supported by the ISRAEL SCIENCE FOUNDATION (grant No. 376/21) at Bar-Ilan university. 

The authors would like to thank Omer Offen for suggesting the problem and for his dedicating help. 
The authors would also like to express their gratitude to Rahul Krishna for the useful discussions.

\section{Notations}
\subsection{General Setup}
Let $F$ be a number field and $\A = \A_F$ its ring of ad\e les. We tend to use the bold font letter as an algebraic variety and the corresponding common font letter as its rational point. If $\mathbf{X}$ is an algebraic variety over $F$, we write $X = \mathbf{X}(F)$ for its $F$-points. Let $E$ be a field extension of $F$, we will write the ring of ad\e les for the number field $E$ as $\A_E$ and the base change of the variety $\mathbf{X}$ to $E$ as $\mathbf{X}_E$.

For an algebraic group $\mathbf{Q}$ defined over $F$ we denote $X^*(\mathbf{Q})$ the lattice of $F$-rational characters of $Q$. Let $\fraka_Q^* = X^*(\mathbf{Q}) \otimes_\Z \R$ and let $\fraka_Q = \Hom_\R(\fraka_Q^*,\R)$ be its dual vector space with the natural pairing $\langle \cdot,\cdot \rangle = \langle \cdot,\cdot \rangle_Q$. We endow $\fraka_Q$ and $\fraka_Q^*$ with Euclidean norms $\|\cdot\|$. Note that since they are finite dimensional vector spaces any of two norms on either of them are equivalent. We denote by $\fraka_\C$ the complexification of a real vector space $\fraka$.

\subsection{Reductive Groups}\label{redgp}
Let $\G$ be a reductive group over $F$. We fix a minimal parabolic subgroup $\P_0$ of $\G$ defined over $F$, a maximal $F$-split torus $\mathbf{T}$ of $\mathbf{G}$ contained in $\mathbf{P}_0$, and a maximal compact subgroup $K$ of $\G(\A)$ where $\mathbf{G}(\A) = \mathbf{P}_0(\A)K$.

Let $\P$ be (any) parabolic subgroup of $\G$. Fix a Levi decomposition $\P = \M \ltimes \U$. Fix a maximal split torus $T_M$. It is a fact that $X^*(P) = X^*(M)$ and both of them are a sublattice of $X^*(T_M)$ of full rank. Consequently, we have $\fraka_M^* = \fraka_P^*$ . We tend to use $\fraka_M^*$ to unify the notations.

Let $\P$ be a semi-standard parabolic subgroup of $\G$. There is a unique Levi decomposition $\P = \M \ltimes \U$ where $\M$ contains the fixed maximal split torus $\mathbf{T}$. Starting from now, unless mentioned, a parabolic subgroup in this paper is semi-standard and the Levi decomposition is the unique such one. If $\Q = \L \times \V$ contains $\P$ (hence $\Q$ also semi-standard), such Levi subgroup $\L$ of $\Q$ contains $\M$. There is a canonical projection $\fraka_M^* \twoheadrightarrow \fraka_L^*$ induced by $T_M \hookrightarrow T_L$, and a canonical inclusion $\fraka_M^* \hookrightarrow \fraka_L^*$ induced by the restriction maps induced by $T_M \hookrightarrow M$ and $T_L \hookrightarrow M$. Define $(\fraka_M^L)^*$ as the kernel of the above projection, we have the split exact sequence:
\begin{align}\label{liealg}
0 \rightarrow (\fraka_M^L)^* \rightarrow \fraka_M^* \rightarrow \fraka_L^* \rightarrow 0.
\end{align}
In particular, all the above spaces can be viewed as subspaces of $\fraka_0^*$.

Taking dual vector spaces for the above split exact sequence, we get one for the Lie algebras. Similarly, all the involved vector spaces are subspaces of $\fraka_0$.



We define the map $H_0 : \mathbf{T}(\A) \rightarrow \fraka_0$ as the one such that
\[
e^{\langle \chi,H_0(t) \rangle} = |\chi(t)|
\]
for all $\chi \in X^*(\mathbf{T})$ and $t \in \mathbf{T}(\A)$. We extend the map $H_0$ to $\G(\A)$ by the Levi decomposition and the Iwasawa decomposition. Let $\P = \M \ltimes \U$ be a parabolic subgroup, we define $H_P = H_M$ as the composition of the projection $\fraka_0 \rightarrow \fraka_M$ and $H_0$.

Let $\rho_P \in \fraka_M^*$ be the unique one such that
\[
\delta_P(p) = e^{\langle 2\rho_P,H_P(p) \rangle}
\]
for all $p \in \P(\A)$.

We also fix a Siegel domain $\mathfrak{S}_G$ for $G\mo G(A)$ and let $\mathfrak{S}_G^1 = \mathfrak{S}_G \cap \G(\A)^1$.

Let $\mathbf{T}_\G$ be the split part of the identity component of the center of $\G$. We embed $\mathbf{T}_\G(\R)$ in $\mathbf{T}_\G(\A)$ by the diagonal embedding $\R \hookrightarrow F_\infty$ and the usual embedding $F_\infty \hookrightarrow \A$ and denote $A_G$ the image of the identity component $\mathbf{T}_\G(\R)^\circ$ in $\mathbf{T}_\G(\A)$.
{\color{black}
We have 
\[H_G:A_G\to\fraka_G\]
and we denote $\nu\mapsto e^{\nu}$ the inverse map of $H_G$.
}
For convenience we denote the adelic quotient of $\G$ as
\[
[\G] = A_GG\mo \G(\A).
\]
More generally, if $\H$ is a closed subgroup of $\G$ defined over $F$, then we denote
\[
A_G^H = A_G \cap \H(\A)
\]
and
\[
[\H]_G = A_G^H H \mo \H(\A) \subseteq [\G].
\]

\subsection{Height Function}
For $n \in \mathbb{N}$ let $\GL_n$ be the general linear group of rank $n$. For a matrix $g = (g_{i,j}) \in \GL_n(\A)$ let
\[
\|g\| = \|g\|_{\GL_n} = \prod_v \max_{1 \leq i,j \leq n} \{|g_{i,j}|_v,|(g^{-1})_{i,j}|_v\},
\]
where the product ranges over all places $v$ of $F$.

Fixing a faithful $F$-rational representation $\rho:\G \rightarrow \GL_n$, we defined 
\[
\|g\| = \|g\|_\rho = \|r(g)\|_{\GL_n(\A)}.
\]
We list some standard facts about the height function $\|\cdot\|$. The proofs can be found in \cite{MW}. We use the notation $A \ll B$ to mean that there exists a constant $c>0$ such that $A \leq cB$, where the constant $c$ is independent of the underlying parameters. If we want to emphasize the dependence of $c$ on other parameters, say, $T$, we will write $A \ll_T B$. In particular, we have
\begin{align}
    1 &\ll \|g\| \text{ for all $g \in \G(\A)$,}\label{5a}\\
    \|g_1g_2\| &\ll \|g_1\|\|g_2\| \text{ for all $g \in \G(\A)$,}\label{5b}\\
    \|H_0(g)\| &\ll 1 + \log\|g\| \text{ for all $g \in \G(\A)$,}\label{5c}\\
    \log \|g\| &\ll 1 + \|H_0(g)\| \text{ for all $g \in \mathfrak{S}_G^1$,}\label{5d}\\
    \|g\| &\ll \|\gamma g\| \text{ for all $g \in \mathfrak{S}_G$ and $\gamma \in G$,}\label{5e}\\
\text{there exists $N$ such that} \|a\|\|g\| &\ll \|ag\|^N \text{for all $g \in \G(\A)^1,a\in A_G$}.\label{2}
\end{align}
Moreover, if $\Omega$ is a compact subset of $\G(\A)$, then we have
\begin{equation}
\sup_{x \in \Omega, g \in \G(\A)}\| H_0(gx)-H_0(g)\| = \sup_{x \in \Omega, k \in K}\|H_0(kx)\| < \infty,
\end{equation}
where $K$ is a maximal compact subgroup of $\G(\A)$.

\subsection{Root System}
Given a reductive group $\G$, we fix a minimal parabolic subgroup $\P_0$ and $\T$ be the maximal split torus of the center of $\P_0$.
We let $\Sigma=R(T,G)$ be the root system of $G$ with respect to $T$ and $\Delta_0 = \Delta_0^G$ be the basis of simple roots with respect to $P_0$. Both $\Sigma$ and $\Delta_0$ are viewed as a subset of $\fraka_0^*$. (Recall that there is a bijection between the collection of subsets of $\Delta_0$ and standard parabolic subgroups of $\G$.) 
We called $\P$ a standard (\textit{resp.} semistandard) parabolic subgroup of $\G$ if $\P\supseteq\P_0$ (\textit{resp.} $\P\supseteq\T$). 
Each semistandard parabolic subgroup admits a unique Levi decomposition $\P=\M\ltimes\U$, where $\M\supseteq\T$. We call $\M$ the Levi subgroup of $\G$ (or the Levi part of $\P$), and $\U$ the unipotent radical of $\P$.
If a Levi subgroup $\M$ contains $\T$, then we call $\M$ a semistandard Levi subgroups, and if a Levi subgroup comes from a Levi decomposition of a standard parabolic subgroup $\P$, then we call $\M$ a standard Levi subgroup. 
In this paper, we assume that all parabolic subgroups and Levi subgroups are standard unless otherwise specified. For any standard parabolic subgroup $\P = \M \ltimes \U$ we have (cf. \cite{LR}, Lemma 6.1.1)
\begin{equation}\label{6}
    \|m\| \ll \|mu\| \text{ for all } m \in \M(\A),u\in\U(\A).
\end{equation}

Given a semistandard parabolic subgroup with Levi decomposition $\P=\M\ltimes\U$, we have $\fraka_P^*=\fraka_M^*$. If $\Q\supseteq\P$ is another larger semistandard parabolic subgroup, we have a unique Levi decomposition $\Q=\L\ltimes\V$ where $\L\supseteq\M$. Thus, $\fraka_L$ is a subspace of $\fraka_M$. Further, there is a canonical decomposition of $\fraka_M=\fraka_L\oplus\fraka_M^L$. We can apply the same setting to the dual spaces. 

We follow the similar notation from \cite[\S I.1.6]{MW}. Given a Levi subgroup $\M$ of $\G$, we define $\Sigma^M=R(T,M)$ to be the subset of $\Sigma$. Let $\T_M$ be the maximal split torus in the center of $\M$, and denote the set of roots of $\G$ relative to $\T_M$ by $\Sigma_M=R(\T_M,\G)$. Let $\Sigma_P$ be the subset of positive roots of $\Sigma_M$ with respect to $\P$. For $\alpha\in\Sigma_M$, we say $\alpha>0$ if $\alpha\in\Sigma_P$ and $\alpha<0$ if $\alpha\notin\Sigma_P$. By the isomorphism of $Hom(\T_M,\G)\otimes\R\cong\fraka_M^*$, we can identify $R(\T_M,\G)$ with a subset of $\fraka_M^*$ and thus we have the restriction map: $R(\T,\G)\to R(\T_M,\G)\cup \{0\}$. We denote $\Delta_M$ to be the set of non-trivial image of $\Delta_0$ under this restriction map. This is a generating set of $\fraka_M^*$.

\subsection{Weyl Group}
Let $W = W^G = N_G(T) / C_G(T)$ be the Weyl group of $G$ with respect to $T$. When $G$ is split and $T$ is a split maximal torus, $C_G(T) =T$. We assume the fixed Euclidean structure on $\fraka_0$ is $W$ invariant, i.e.
\[
\|w x\| = \| x \|, \ \forall w \in W \text{ and } x \in \fraka_0.
\]
By definition, elements of $W$ are $C_G(T)$-cosets in $N_G(T)$. Hence $w=nC_G(T)$ for some $n \in N_G(T)$ and we denote $w \in n$ in this case.
We define the following notations:
\begin{itemize}
 \item for a Levi subgroup $M$ let 
 \[
 \ _MW_M := \{w \in W \mid \ell(w) \leq \ell(w_1ww_2)\ \forall w_1,w_2 \in W^M \};
 \]
 \item for two Levi subgroups $M$ and $M^\prime$ we write 
 \[
 W(M,M^\prime):= \{w \in W \mid \ell(w) \leq  \ell(wu)\ \forall u \in W^M;wMw^{-1} = M^\prime\}.
 \]
 \item $$W(M)=\bigcup_{M' \text{ standard Levi of } G} W(M,M')$$
\end{itemize}

We have the following properties:
\begin{enumerate}
 \item if $w \in W(M,M^\prime)$, then $w^{-1} \in W(M^\prime,M)$;
 \item if $w_1 \in W(M_1,M_2)$ and $w_2 \in W(M_2,M_3)$, then $w_2w_1 \in W(M_1,M_3)$;
 \item $W(M,M)$ is a subgroup of $W$, which we can identify with $N_G(M) / M$.
\end{enumerate}

For any Levi subgroups $M \subset L$ we denote by $w_M^L$ the unique element of maximal length in $W(M) \cap W^L$. When $M=M_0$, we denote $w_0^L := w_{M_0}^L$.

\begin{lem}(Bruhat decomposition)\label{Bruhat decomposition}
    Let $P$ be a parabolic subgroup of $G$ and $M$ its Levi subgroup. The inclusion of $\MWM\to W$ induces a bijection between $\MWM\cong W^M\backslash W/W^M$.
\end{lem}

Let $\alpha\in\Sigma$ be a root and we denote $s_\alpha$ be its simple reflection. We remark that the Weyl group of $G$ is generated by $s_\alpha:\alpha\in\Delta$. Let $\alpha\in\Delta_M$, by the construction in \cite[I.1.7]{MW}, we denote $s_\alpha$ be the elementary symmetry of $\alpha$. Although simple reflection and elementary symmetry share the same symbol, it will be clear from the context. Additionally, we will clarify in the paper which type of $s_\alpha$ we are referring to.

\subsection{Some Function Spaces}\label{secfunc}
Let $(V,\|\cdot\|)$ be a Euclidean space and $R>0$. We denote
\[
C_R(V) = \{f:V \rightarrow \C \mid f(v) \ll e^{-R\|v\|} \text{ for all } v \in V\}.
\]
Let $V^*$ be its dual vector space and $V^*_\C = V^* \otimes_\R \C$. Let $r>0$ and let $V^*_{\C,r} = \{\lambda \in V_\C^* \mid \|\mathrm{Re}~\lambda \| <r\}$. We denote
\[
P^r(V^*) = \{ \phi: V^*_{\C,r} \rightarrow \C \mid \sup\limits_{\lambda \in V_{\C,r}^*}|\phi(\lambda)|(1+\|\lambda\|)^N < \infty, N = 1,2,\ldots \}.
\]
Further we denote
\[
P^r(V^*;W) = P^r(V^*) \otimes_\C W.
\]
when $W$ be a finite dimensional $\C$-vector space.

\begin{lem}\label{pw}
Let $f: V \rightarrow \C$. The following conditions are equivalent.
\begin{enumerate}
    \item For all $r < R$ and a differential operator $D$ on $V$ with constant coefficients, $Df \in C_r(V)$.
    \item For all $r < R$ the function $f(v)e^{r\sqrt{1+\|v\|^2}}$ is a Schwartz function on $V$.
    \item The Fourier transform
    \[
\hat{f}(\lambda) = \int_V f(v) e^{\pair{\lambda}{v}} dv
    \]
    of $f$ admits holomorphic continuation to $\{ \lambda \in V_\C^*: \| \mathrm{Re} \lambda \| <R \}$ and belongs to $\bigcap\limits_{r<R} P^r(V^*)$.
\end{enumerate}
\end{lem}
\begin{proof}
See \cite{LO}, Lemma 2.1.
\end{proof}
Fix a parabolic subgroup $\P = \M \ltimes \U$ of $\G$. For any $f \in C_R(\fraka_0^M)$ we define
\begin{equation}\label{theta}
\theta_f^M(g) = \sum_{\gamma \in P_0 \cap M \mo M} e^{\pair{\rho_0}{H_0(\gamma g)}} f(H_0^M(\gamma g)), g \in \G(\A).
\end{equation}
Given the convergence of the series, it is immediate that
\begin{enumerate}
    \item $\theta_f^M$ is defined on $A_G \U(\A) M \mo \G(\A)$;
    \item $\theta_f^M$ is right $K$-invariant;
    \item $\theta_f^M(ag) = e^{\pair{\rho_P}{H_0(a)}}$ for all $a \in A_M$ and $g \in \G(\A)$.
\end{enumerate}


\begin{lem}\label{lem2.2}
There exists $R_0>0$, for all $R>R_0$, the sum defining $\theta_f^M$ in (\ref{theta}) is absolutely convergent for any $f \in C_R(\fraka_0^M)$. Moreover, for any $N>0$ there exists $R_0$ and $N^\prime$, for all $R>R_0$ and $f \in C_R(\fraka_0^M)$ we have
\[
\sup\limits_{m \in \mathfrak{S}_M^1} |\theta_f^M(mg)|\|m\|^N \ll_{N,f} \|g\|^{N^\prime}, g \in \G(\A).
\]
\end{lem}

\begin{proof}
We will prove the second part directly and the first part will follow. The proof is essentially given in \cite[Lemma 2.2]{LO} but we write it in richer details.

First, we reduce the assertion to the case when $g = p \in \P(\A)$.
Suppose for given $N,R,N^\prime,f$ and $C$ we have
\[
\sup_{m \in \mathfrak{S}_M^1}|\theta_f^M(mp)|\|m\|^N \leq C\|p\|^{N^\prime},\text{ for all } p \in P(\A).
\]
Let $g \in G(\A)$ and write $g = pk$. We have
\begin{align*}
\sup_{m \in \mathfrak{S}_M^1}|\theta_f^M(mg)|\|m\|^N =& \sup_{m \in \mathfrak{S}_M^1}|\theta_f^M(mpk)|\|m\|^N \\
=& \sup_{m \in \mathfrak{S}_M^1}|\theta_f^M(mp)|\|m\|^N \\
\leq& C \|p\|^{N^\prime}.
\end{align*}
By assumption and (\ref{5b}), the above is bounded by
\begin{align*}
C \|gk^{-1}\|^{N^\prime}
\leq CC_1 \|g\|^{N^\prime} \|k\|^{N^\prime}
\leq CC_1 \sup_{k \in K} \|k\|^{N^\prime} \|g\|^{N^\prime}. 
\end{align*}
Note that both $C_1$ and $\sup_{k \in K} \|k\|^{N^\prime}$ are independent of $g$.

Second, we reduce the assertion to the case when $p = m_0 \in M(\A)$. Suppose for given $N,R,N^\prime,f$ and $C$ we have
\[
\sup_{m \in \mathfrak{S}_M^1}|\theta_f^M(mm_0)|\|m\|^N \leq C\|m_0\|^{N^\prime}, \text{ for all }  m_0 \in G(\A).
\]
Let $p \in P(\A)$ and write $p = m_0 u^\prime$. We have:
\begin{align*}
\sup_{m \in \mathfrak{S}_M^1}|\theta_f^M(mp)|\|m\|^N =& \sup_{m \in \mathfrak{S}_M^1}|\theta_f^M(mm_0 u^\prime)|\|m\|^N \\
=& \sup_{m \in \mathfrak{S}_M^1}|\theta_f^M(mm_0)|\|m\|^N \\
\leq& C \|m_0\|^{N^\prime}\\
\leq& CC_1 \|p\|^{N^\prime} 
\end{align*}
where the third line is our assumption and the last one follows from (\ref{6}).

Third, we reduce the assertion to the case when $m^\prime \in M(\A)^1$. Suppose for given $N,R,f$ and $C$ we have
\[\sup_{m \in \mathfrak{S}_M^1}|\theta_f^M(mm^\prime)|\|m\|^N \leq C\|m^\prime\|^{N}, \text{ for all } m^\prime \in M(\A)^1 .
\]

Let $m_0 = am^\prime$ where $a \in A_M$, we have
\begin{align*}
\sup_{m \in \mathfrak{S}_M^1}|\theta_f^M(mm_0)|\|m\|^N =& \sup_{m \in \mathfrak{S}_M^1}|\theta_f^M(amm^\prime)|\|m\|^N &\\
=& \sup_{m \in \mathfrak{S}_M^1} e^{\langle \rho_0, H_0(a) \rangle}|\theta_f^M(mm^\prime)|\|m\|^N &\\
\leq& Ce^{\|\rho_0\|}\|a\|^{\|\rho_0\|}\|m^\prime\|^N.
\end{align*}
Since $\|a\|$ has a lower bound, we raise the powers from $N$ to $N_1$ and the above is controlled by
\begin{align*}
C \|a\|\|m^\prime\|^{N_1}
\leq  CC_1 \|am^\prime\|^{N_1N_2}
\end{align*}
for some $C_1$ and $N_2$.

Then we reduce the assertion to the case for the neutral element $e$.
Suppose that for given $N,R,f$ and $C$ we have
\[
\sup_{m \in \mathfrak{S}_M^1}|\theta_f^M(m)|\|m\|^N \leq C.
\]
For any $m^\prime \in M(\A)^1$, and any $m \in \mathfrak{S}_M^1$, there exists $\gamma \in M(k)$ and $m_1 \in \mathfrak{S}_M^1$, such that $mm^\prime = \gamma m_1$. Moreover, By (\ref{5b}) and (\ref{5e}),
there exists $C_1$ such that:
\[
\|m\| \leq C_1 \|m_1\|\|m^\prime\|.
\]
We also observe that $\theta_f^M(mm^\prime) = \theta_f^M(m_1)$. Hence for any $m \in \mathfrak{S}_M^1$ and $m^\prime \in M(\A)^1$:
\begin{align*}
|\theta_f^M(mm^\prime)|\|m\|^N \leq& C_1^N |\theta_f^M(m_1)|\|m_1\|^N\|m^\prime\|^N\\
\leq& \sup_{m_1 \in \mathfrak{S}_M^1} C_1^N |\theta_f^M(m_1)|\|m_1\|^N\|m^\prime\|^N\\
\leq& C_{new} \|m^\prime\|^N.
\end{align*}
Hence
\[
\sup_{m \in S_M^1} |\theta_f^M(mm^\prime)|\|m\|^N \leq C \|m^\prime\|^N.
\]

Lastly, we claim that for all $N \geq 0$, there exists $R>0$, for all $f \in C_R(\fraka_0^M)$, there exists $C>0$, for all $m \in \mathfrak{S}_M^1$, we have
\[
|\theta_f^M(m)|\|m\|^N \leq C.
\]
We will make use of the following fact:
$\#\{\gamma \in P_0 \cap M \mo M: \|H_0^M(\gamma m)\| \leq X\} \leq C(e^X + \|m\|)^{N_1}$ for all $X \geq 0$ and $m \in \mathfrak{S}_M^1$.
First note that (\ref{5d}) and (\ref{2}) imply that there exists $C_1$ and $C_2$, such that
\begin{align}
\|m\| \leq C_1e^{C_2\|H_0^M(m)\|}, \text{ for all } m \in S_M^1, \gamma \in M(k).
\end{align}
Hence we have
\begin{align*}
|\theta_f^M(m)| \|m\|^N \leq & \sum_{\gamma \in P_0 \cap M \mo M} e^{\langle \rho_0, H_0(\gamma m) \rangle} |f(H_0^M(\gamma m)| \|m\|^{-N^\prime} \|m\|^{N+N^\prime}\\
\leq & \sum_{\gamma \in P_0 \cap M \mo M} e^{(\| \rho_0\|-R +C_2) \| H_0(\gamma m)\|} \|m\|^{N^\prime}\\
\leq & \sum_{X=0}^{+\infty}\sum_{\{\gamma|X \leq \|H_0^M(\gamma m)\| \leq X+1\}}e^{(\| \rho_0\|-R +C_2) \| H_0(\gamma m)\|}.
\end{align*}
Let $R$ be such that $R> \|\rho_0\|+N^\prime +C_2$ we have the above
\begin{align*}
\leq & \sum_{X=0}^{+\infty} (\frac{e^{X+1}}{\|m\|} + 1)^{N^\prime} e^{(\|\rho_0\|-R)X}.
\end{align*}
By expanding $(\frac{e^{X+1}}{\|m\|} + 1)^{N^\prime}$ via binomial theorem, we get an upper bound independent of $m$, where we notice that $\|m\|$ admits a lower bound.
\end{proof}

\begin{lem}\label{lem2.3}
For any $N>0$ there exists $R>0$ such that
\[
\sup_{g \in \mathfrak{S}_G^1} \sum\limits_{\gamma \in {P \mo G}} |\phi(\gamma g)| \|g\|^N < \infty
\]
and, in particular,
\[
\sup_{g \in \G(\A)} \sum\limits_{\gamma \in {P \mo G}} |\phi(\gamma g)| < \infty,
\]
for any function $\phi$ on $A_G\U(\A)M \mo \G(\A)$ satisfying
\begin{equation}\label{condition phi}
\sup_{m \in \mathfrak{S}_M^1,a \in A_M, k \in K} \delta_P(a)^{-\frac{1}{2}}|\phi(amk)|\|m\|^t e^{R\|H_P^G(a)\|} < \infty,\ t = 1,2,3,\ldots.
\end{equation}
\end{lem}
\begin{proof}
We show for $N$ large, which is enough. Let $f(v) = e^{-R\|v\|}$ for $v \in \fraka_0^G$. The condition (\ref{condition phi}) together with (\ref{5c}) implies
\[
|\phi(g)| \ll_{\phi,R} e^{\pair{\rho_P}{H_P(g)}}f(H_0^G(g)),\ g \in \mathfrak{S}_M K.
\]
It follows that
\[
|\phi(g)| \ll_{\phi,R} \sum_{\gamma \in P_0 \cap M \mo M} e^{\pair{\rho_0}{H_0(\gamma g)}} f(H_0^G(\gamma g)),\ g \in \G(\A).
\]
Therefore,
\[
\sum_{\gamma \in P \mo G} |\phi(\gamma g)| \ll_{\phi,R} \sum_{\gamma \in P_0 \mo G} e^{\pair{\rho_0}{H_0(\gamma g)}} f(H_0^G(\gamma g)),\ g \in \G(\A).
\]
The lemma now follows from Lemma \ref{lem2.2} with $\M=\G$.
\end{proof}

Let $\Au_P^{mg}(G)$ be the space of continuous functions $\varphi$ on $\U(\A)M \mo \G(\A)$ where
\begin{enumerate}
    \item[(1)] $\varphi$ is of moderate growth, i.e. there exists $c,r \in \R$ such that for all $g \in \G(\A)$, we have $|\varphi(g)| \leqslant c \|g\|^r$;
    \item[(2)] $\varphi(ag) = e^{\pair{\rho_P}{H_0(a)}}$ for all $a \in A_M, g \in \G(\A)$.
\end{enumerate}
Let $\Au_P^{rd}(G)$ be the subspace of $\Au_P^{mg}(G)$ consisting of $\varphi$ such that for all $N>0$
\[
\sum\limits_{m \in \mathfrak{S}_M^1, k \in K} |\varphi(mk)|\|m\|^N < \infty.
\]
For $\varphi \in \Au_P^{mg}(G)$ and $\lambda \in \fraka_{M,\C}^*$ let
\[
\varphi_\lambda(g) = e^{\pair{\lambda}{H_P(g)}}\varphi(g), g \in \G(\A).
\]

For any $R>0$ let $C_R(\U(\A)M \mo \G(\A))$ be the space of continuous functions $\phi$ on $A_G \U(\A) M \mo \G(\A)$ satisfying (\ref{condition phi}) such that $\phi(\cdot g)$ is a cuspidal function on $M \mo \M(\A)$ for all $g \in \G(\A)$. For $R>0$ and $\phi \in C_R(U(\A)M \mo \G(\A))$ we define
\[
\theta_\phi(g) = \sum_{\gamma \in P \mo G} \phi(\gamma g).
\]
By Lemma \ref{lem2.3}, the series converges for large enough $R$. For any $\lambda \in \fraka_{M,\mathbb{C}}^*$ with $\| \mathrm{Re} \lambda \| < R$ we write
\[
\phi[\lambda](g) = e^{-\pair{\lambda}{H_P(g)}}\int_{A_G \mo A_M} e^{-\pair{\lambda+\rho_P}{H_P(g)}}\phi(ag) da.
\]
We have $\phi[\lambda] \in \Au_P^{rd}(G)$.

Let $C_R^\infty(\U(\A)M\mo \G(\A)$ be the subspace of smooth functions in $C_R(\U(\A)M\mo \G(\A)$. For $\phi \in C_R^\infty(\U(\A)M\mo \G(\A)$, we have
\[
\phi(g) = \int_{\lambda_0 + i(\fraka_M^G)^*} \phi[\lambda]_\lambda(g) d\lambda
\]
for any $\lambda_0 \in (\fraka_M^G)^*$ with $\|\lambda_0\| <R$. Moreover, it easily follows from Lemma 2.1 and its proof that for any $R^\prime <R$ and $N>0$ we have
\begin{equation}\label{(10)}
\sup_{m \in \mathfrak{S}_M^1,k\in K, \lambda \in (\fraka_M^G)_\C^*:\|\mathrm{Re} \lambda\| \leq R^\prime} |\phi[\lambda](mk)|(\|m\|+\|\lambda\|)^N < \infty.
\end{equation}
Thus, we may think of $\phi \in C_R^\infty(\U(\A)M \mo \G(\A))$ as a holomorphic map on $\{\lambda \in (\fraka_M^G)_\C^*:\| \mathrm{Re} \lambda \| <R \}$ with values in $\Au_P^{rd}(G)$ satisfying (\ref{(10)}).

\subsection{Unitary Groups and Symplectic Groups}
Let $F$ be a number field and $E/F$ be a quadratic extension.
Let $U_{2n}$ be the quasi-split unitary group over $F$ with respect to $E/F$ of rank $n$ and we fix the embedding
\[
\Un{2n} = \{g \in \ResGL{2n} \mid {}^t\bar{g}J_n g = J_n\},
\]
where
\[
J_n = \begin{pmatrix} 0 & w_n \\ -w_n & 0 \end{pmatrix}
\]
and $w_n = (\delta_{i,n+1-j}) \in \GLn{n}$ is the permutation matrix with ones on every entry of the anti-diagonal and zeros elsewhere. 
Let $\Spn{2n}$ be the fixed points under the Galois action, i.e. 
\begin{align*}
\Spn{2n} &= \{h \in \Un{2n}| h = \bar{h}\}\\
&= \{g \in \GLn{2n} \mid {}^t{g}J_n g = J_n\}.
\end{align*}

Let $^*: \GLn{n}\to\GLn{n}$ be the automorphism given by $g\mapsto g^*=w_n\;^tg^{-1} w_n$. The embedding $\imath:\ResGL{n}\to\Un{2n}$, $g\mapsto \diag(g,g^*)$ identifies $\ResGL{n}$ with the Siegel Levi subgroup of $\Un{2n}$.

We take $\B_n$ as the upper triangular Borel of $\Un{2n}$, this subgroup admits a Levi decomposition $\B_n=\T_n\ltimes \N_n$, where $\T_n$ is the diagonal maximal split torus, i.e.
\[
\T = \{\diag(a_1,\ldots,a_n,a_n^{-1},\ldots,a_1^{-1}) \mid a_i \in \mathbb{G}_m \forall\ i = 1,2,\ldots,n\}.
\]
and $\N_n$ is the subgroup of upper unitriangular matrices in $\B_n$. The standard parabolic subgroups of $\G$ are classified by its root system $R(T,G)$. For every tuple $\gamma = (n_1,n_2,\ldots,n_k;r)$ where $k,r\geq 0$, $n_1,\ldots,n_k > 0$, and $\sum\limits_{i=1}^k n_i + r = n$, we associate the standard parabolic subgroup $\P = \P_\gamma = \M_\gamma \ltimes \U_\gamma$ 
where $\M_\gamma \cong \prod\limits_{i=1}^k\ResGL{n_i}\times \Un{2r}$. The isomorphism is given by the map 
\begin{align*}
    \imath:  \prod_{i=1}^k\ResGL{n_i}\times \Un{2r}&\to \M_\gamma\\
    (g_1,\cdots,g_k;h)&\mapsto \mathrm{diag}(g_1,\ldots,g_k,h,g_1^*,\ldots,g_k^*).
\end{align*}

\subsection{The Symmetric Space}
Let $\G$ be a connected reductive group and $\theta$ an involution on $\G$. We associate the symmetric space $$\X=\X(\G,\theta)=\{g\in \G: g\theta(g)=e\}$$
with the $\G$-action by $\theta$-twisted conjugation
\[
(g,x) \mapsto g\cdot x = gx\theta(g)^{-1},\quad g\in \G,x\in \X.
\]

We follow the set up in \cite{LO}, except for that we denote $2n$, instead of $n$, as the index of our unitary groups and symplectic groups.

Fix $n \in \mathbb{N}$. Let $G= \mathbf{U}_{2n}$. Let 
\[
\X=\{g \in \G \mid g \bar{g} = e\}
\]
with the $\G$-action by $\theta$-twisted conjugation
\[
(g,x) \mapsto g\cdot x = gx\bar{g}^{-1},\quad g\in \G,x\in \X.
\]
For every $x \in \X$ and $\Q$ subgroup of $\G$, denote $\Q_x = \mathrm{Stab}_\Q(x)$, an algebraic group defined over $F$. 

\begin{example}
\begin{enumerate}\label{two examples of single orbit}
    \item Let $\G=\GLn{2n}$ and $\theta(g)=J_n {}^tg^{-1}J_n^{-1}$. In this case, $\mathbf{G_e}=\Spn{2n}$ and $\X=\{g\in \G:gJ_n=J_n{}^tg\}$. It is well known that $X=G\cdot e$.\label{GL_2n, twist by J_n is single orbit.}
    \item Let $\G=\ResGL{n}$ and $\theta(g)=\bar g$. In this case, $G=\operatorname{GL}_n(E)$ and $G_e=\operatorname{GL}_n(F)$. By Hilbert theorem 90, we have $X=G\cdot e$.
\end{enumerate}
\end{example}

In this paper, we consider the symmetric space $\X=\X(\Un{2n},\theta)$ where $\theta(g)=\bar g$. In particular, we denote $\mathbf{H} = \G_e = \mathbf{Sp}_{2n}$.

We refer a lemma from \cite[\S 3.2, Lemma 4]{MO}.

\begin{lem}\cite[\S 3.2 Lemma 4]{MO}\label{X(U_2n,theta) is single G orbit.}
For $\X = X(\U_{2n},\theta)$, where $\theta(g)=\bar g$, we have
\[
X = U_{2n}\cdot e.
\]
Hence $X \cong G/H$.
\end{lem}

We take $\P_0$ to be the Borel subgroup $\mathbf{B} = \mathbf{B}_{2n}$. We take the maximal split torus $\mathbf{T}$ in $\mathbf{B}$ as the diagonal one as above. The Lie algebra $\fraka_T^*$ is naturally identified with $\R^{2n}$ and the identification is an isomorphism of $\R$-vector space isomorphism. Note that $\fraka_T^* = \fraka_{M_0}^*$ and we also denote them as $\fraka_0^*$ according to our general notions related to reductive groups.

Let $\gamma = (n_1,\ldots,n_k;r)$ with $\sum_{i=1}^k n_i + r = n$. For $\M = \M_\gamma$ the space $\fraka_M^* \cong \R^k$ is embedded in $\fraka_T^* \cong \R^{n}$ as elements of the form
\[
(\overbrace{\lambda_1,\ldots,\lambda_1}^{n_1},\ldots,\overbrace{\lambda_k,\ldots,\lambda_k}^{n_k},\overbrace{0,\ldots,0}^{r}).
\]

Under the identification $\fraka_T^* \cong \R^{n}$ we have $\Sigma = \{\pm e_i \pm e_j:1\leq i\leq j\leq n\}$, where $\{e_i\mid 1 \leq i \leq 2n\}$ is the standard basis of $\R^{2n}$. Also, $\Delta_0 = \{\alpha_1,\ldots,\alpha_{2n}\}$, where $\alpha_i = e_i - e_j$ when $i \neq 2n$ are the short simple roots, and $\alpha_{2n} = 2 e_{2n}$ is the long simple root.

\section{Orbit analysis}
\subsection{Involutions on the Weyl Group}
Let $W$ be a Weyl group of a group $G$ with basis $\Delta$ of simple roots. We denote the set of involutions in $W$ by $W[2]=\{w\in W:w^2=e\}$. 
\begin{defn}
    An involution $w\in W[2]$ is called minimal if there exists a Levi subgroup $M$ of $G$ such that $w=w_0^M$, the longest element of $W_M$, and $w\alpha=-\alpha$ for all $\alpha\in \Delta_0^M$.
\end{defn}

\begin{example}\label{min inv of S_n}
    The minimal involution in $S_n$, which is the Weyl group of $\GLn{n}$, with the standard basis $\Delta=\{e_i-e_{i+1}\}$, are the products of disjoint simple reflections. That is, $w$ is minimal in $S_n$ if and only if $w=s_{i_1}\cdots s_{i_k}$ where each pair $s_{i_u}$ and $s_{i_v}$ are disjoint transpositions of the form $s_{i}=(i,i+1)$ for $1\leq u\neq v\leq k$. 
\end{example}

Let $\mathfrak{W}_n$ be the signed permutation group in $n$ variable. We realize it as
\[
\mathfrak{W} = \mathfrak{W}_n = S_n \ltimes \Xi_n
\]
where $\Xi_n$ is the group of subsets of $\{1,\ldots,n\}$ with symmetric difference as multiplication. The action of $S_n$ on $\Xi_n$ is given by
\[
\tau \cdot \setc  =\tau\setc\tau^{-1}, \ \tau \in S_n,\ \setc \in \Xi_n.
\]
It is a classical result from Lie algebra that the Weyl group $W = W(G,T) = \langle s_{\alpha_1},\ldots,s_{\alpha_n} \rangle$ of $\U_{2n}$ is isomorphic to the signed permutation group $\mathfrak{W}_n$, under the identification
\begin{enumerate}
 \item $\tau(e_i) = e_{\tau_i}$,
 \item $\setc(e_i) = \begin{cases} e_i,\ i \notin \setc\\ -e_i, \ i \in \setc \end{cases}$.
\end{enumerate}

We observe that $$\frakW_n[2]=\{\tau\frakc:\tau\in S_n, \tau(\frakc)=\frakc\}.$$
With the understanding of the structure of semidirect product, we can further prove that the set of minimal involutions in $\frakW_n$ is
\begin{align}\label{min inv of U_2n}
    \{\tau\frakc_{k,n}: 0\leq k \leq n,\; \tau \text{ is a minimal involution of }S_k,\; c_{k,n}=\{k+1,\cdots,n\}\}.
\end{align}
Here $S_k$ embedded into $S_n$ via fixing $c_{k,n}$ point-wisely.

Let $w = \tau \setc \in \frakW_n$. Define the following subsets of $\{1,\ldots,n\}$:
\begin{itemize}
 \item $\setc_+(w) = \{i \in \setc \mid \tau(i) = i\}$;
 \item $\setc_-(w) = \{i \notin \setc \mid \tau(i) = i\}$;
 \item $\setc_{\neq}(w) = \{i \mid \tau(i) \neq i\}$;
 \item $\setc_<(w) = \{i \mid i < \tau(i)\}$.
\end{itemize}
Note that $\setc_{\neq}(w)$ and $\setc_<(w)$ depend only on $\tau$.

\begin{example}
    Let $w=\tau\frakc_{k,n}$ be a minimal involution in $\frakW_n$. Then, $\frakc_+(w)=\frakc_{k,n}$. Also, according to \ref{min inv of S_n} and \ref{min inv of U_2n}, we have a partition
    $$\{1,\cdots, k\}=\frakc_-(w)\sqcup \bigsqcup_{i\in\frakc_<(w)}\{i,i+1\}$$
    and 
    $$\tau=\prod_{i\in\frakc_<(w)} s_i$$
    where $s_i$ is defined in Example \ref{min inv of S_n}.
\end{example}

\subsection{P-Orbits}

Let $P$ be a parabolic subgroup of $G$, and let $x\in X$. We define the map $$\imath_P: P\backslash X\to W[2]\cap \MWM$$ via the Bruhat decomposition (\ref{Bruhat decomposition}). To be more precise, $[x]_P\subseteq PxP$ and by the Bruhat decomposition, there exist a unique element $w\in\MWM$ such that $PxP=PwP$ and we let $\imath_P([x]_P)=w$. For any $x\in X$, $x\bar{x}=e$ and also $P$ is stable under the Galois conjugation, we deduce $PxP=Px^{-1}P$. If $w\in\MWM$ is the element such that $PxP=PwP$, then $x^{-1}\in Pw^{-1}P$ and thus $Px^{-1}P=Pw^{-1}P$. So, $PwP=PxP=Px^{-1}P=Pw^{-1}P$. By the uniqueness of Bruhat decomposition, we have $w=w^{-1}$ and $\imath_P$ is well-defined.

\begin{lem}\label{intersection of B orbit and N_G(T) is non-empty}
There is a bijection between the set of $B$-orbits in $X$ and the set of $T$-orbits in $N_G(T)$ via the map $[x]_B\mapsto N_G(T)\cap [x]_B$.
\end{lem}

We refer to \cite[Lemma 3.2]{LO} for the proof of this lemma. 

\subsection{M-admissible Orbits}
Let $\P=\M\ltimes \U$, and denote $\pr_M:\P\to\M$ the projection to the Levi part of $\P$. Given two parabolic subgroup $\Q_i=\L_i\ltimes\V_i$, $i=1,2$, $\pr_{L_1}(\Q_1\cap\Q_2)$ is a parabolic subgroup of $\L_1$. For $w\in {}_{M}W_M $ we define $\P(w)=\pr_M(\P\cap w\P w^{-1})=\M\cap w\P w^{-1}$. Note that $\P(w)$ is a standard parabolic subgroup of $\M$ with Levi decomposition 
\begin{align*}
    \P(w)=(\M\cap w\M w^{-1})\ltimes(\M\cap w\U w^{-1}).
\end{align*}
We denote $\M(w):=\M\cap w\M w^{-1}$ and $\U(w):=\M\cap w\U w^{-1}$. 
By Bruhat decomposition, given any $g\in G$, there exist a unique $w\in {}_MW_M$ such that $w$ and $g$ are in the same double $P$ coset, i.e. $PwP=PgP$. Let $p\in P$ be such that $g\in pwP$. Then we have
$$\P\cap g\P g^{-1}=p(\P\cap w\P w^{-1})p^{-1},$$
and thus
$$\pr_M(\P\cap g\P g^{-1})=\pr_M(p)\P(w)\pr_M(p)^{-1}.$$
\begin{defn}
    The following statements are equivalent:
    \begin{enumerate}
        \item $\pr_M(\P\cap g\P g^{-1})=\M$,
        \item $(\P\cap g\P g^{-1})\U=\P$,
        \item $\P(w)=\M$,
        \item $\M\cap w\M w^{-1}=\M$,
        \item $\M\cap w\U w^{-1}=1$,
        \item $w\M\subseteq \N_\G(\M)$.
    \end{enumerate}
    If one of these conditions is satisfied, we say $g\in G$ is $M$-admissible.
\end{defn}

We observe that the $M$-admissibility of the element $g\in G$ depends solely on the double coset $PgP$. The fact $PgP=P\Bar{g}P=PwP$ indicates that the condition $\pr_M(\P\cap g\P\bar{g}^{-1})=M$ also satisfies $M$-admissibility, and $g$, $\bar{g}$ lie in the same orbit.

\begin{lem}\cite[Lemma 3.6]{LO}\label{Lem 3.6}
    Let $x\in X$ and $w=\imath_P(x)$. $wM(w)\cap [x]_P$ is non-empty. 
\end{lem}

\begin{proof}
    Since $\imath_P(x)\in W[2]$, so $w^2=1$. Let $w'\in W$ be the element of minimal length in the image of $[x]_P\cap N_G(T)$ under the quotient map $N_G(T)\to W$ (domain is non-empty by \ref{intersection of B orbit and N_G(T) is non-empty}). Noticed that ${(w')}^2=1$ and $PwP=Pw'P$. We claim that there exists a reduced expression $w'=w_1w''ww_2$, with $w_1^{-1}$, $w_2\in W^{M}$ both left $M(w)$-reduced and $w'' \in W^{M(w)}$.

    We claimed that any element in the preimage of $w'$ will contain in $wM(w)$. (See \cite[Lemma 3.6]{LO}).
\end{proof}

\begin{lem}\cite[Lemma 3.7]{LO}\label{lem 3.7}
    Let $w\in \MWM$, $x\in wM(w)\cap X$ and let $\RR(x)$ be the unipotent radical of $\P_x$. Then $\U(w)$ is normal subgroup of $\pr_M(\P_x)$ and $\U(w)\subseteq \pr_M(\RR(x))$. 
\end{lem}

\begin{proof}
    From the proof of Lemma \ref{Lem 3.6}, $w^2=1$. Since $x\in wM(w)$, $x\P x^{-1}=w\P w^{-1}$ and thus $\P_x\subseteq\P\cap x\P x^{-1}=\P\cap w\P w^{-1}$. Therefore, $\pr_M(\P_x)\subseteq\P(w)$. By the Levi decomposition, we know $\U(w)$ is a normal subgroup of $\P(w)$ and thus it is a normal subgroup of $\pr_M(\P_x)$.

    Let $\Z=\U(w)(\U\cap w\P w^{-1})$ then by the Levi decomposition, $\P\cap x\P x^{-1}=\M(w)\ltimes\Z$. We claim that $\P_x=\M(w)_x\ltimes \Z_x$. It is enough to show $\P_x\subseteq\M(w)_x\ltimes \Z_x$. If $p\in \P_x$, then $p=mn$, $m\in \M(w)$, $n\in \Z$. But since $p=x\Bar{p}x^{-1}=x\Bar{m}x^{-1}x\Bar{n}x^{-1}=mn$, by Levi decomposition, $m=x\Bar{m}x^{-1}$ and $n=x\Bar{n}x^{-1}$. It follows that $m\in \M(w)_x$ and $n\in \Z_x$. So, $\R(x)=\Z_x$.

    We now show $\U(w)\in \pr_M(\R(x))$. Let $u\in \U(w)$, and let $v=x\bar{u}\bar{x}$. Recalled that $x\in \X$ implies $x^{-1}=\bar{x}$. We have $v\in\U\cap w\P w^{-1}\subseteq \Z$. Consider the commutator $z=[v^{-1},u^{-1}]$. Since $u\in \M$ and $v\in\U$, we have $u^{-1}vu\in \U$, ($\P=\M\ltimes\U$, semidirect product and $\M$ act on $\U$ by conjugation.) and thus $z=v^{-1}u^{-1}vu\in\U$. Furthermore, since $z=[v^{-1},u^{-1}]=x\bar{u}^{-1}\bar{x}u^{-1}x\bar{u}\bar{x}u$, and $z^{-1}=u^{-1}x\bar{u}^{-1}x^{-1}ux\bar{u}\bar{x}$, we have $z^{-1}=x\bar{z}\bar{x}$. So, $z\in\U':=\U\cap w\U w^{-1}$. Consider the involution on $\U'$: $\theta'(g)=xg\bar{x}$. Since $\U'$ is a unipotent group, we have trivial first cohomology group $H^{1}(\langle\theta'\rangle,\U')=1$\cite[Lemma 0.6]{HW}. Since $z$ satisfies the cocycle condition $z\theta'(z)=1$, $z$ should also be a coboundary, that is, there exist $u'\in\U'$ such that $z=u'\theta'(u')^{-1}=u'x\bar{u'}^{-1}\bar{x}=[v^{-1},u^{-1}]$. By expanding the defining equation for $v$, we have $uvu'=x\overline{uvu'}x^{-1}$, and this implies $uvu'\in \Z_x$. But since $vu'\in \U$, we have $\pr_M(uvu')=u$ and the claim follows.
\end{proof}

Let $x\in X$. Recall that $\pr_M(\P\cap x\P x^{-1})$ is a parabolic subgroup of $\M$. Define $\U(x)$ be the unipotent radical of $\pr_M(\P\cap x\P x^{-1})$ and $\RR(x)$ be the unipotent radical of $\P_x$.

\begin{lem}\cite[Lemma 3.8]{LO}\label{lem3.8}
    Let $x\in X$.
    \begin{enumerate}
        \item The kernel of $\pr_M:\P_x\to \M$ is contained in $\RR(x)$.
        \item $\U(x)$ is a normal subgroup of $\pr_M(\P_x)$ and $\U(x)\subseteq \pr_M(\RR_x)$.
        \item Let $\chi$ be a character of $\P_x(\A)^1\backslash\P_x(\A)$. Then for every function $f:\U(\A)M\backslash\P(\A)\to\C$ satisfying 
        \begin{align*}
            \int_{U(x)\backslash\U(x)(\A)}f(up)du=0, \text{ for all }p\in\P(\A),
        \end{align*}
        we have
        $$\int_{P_x\backslash\P_x(\A)}F(p)\chi(p)dp=0.$$
    \end{enumerate}
\end{lem}

\begin{proof}
    For part 1, we consider the map $\pr_M:\P_x\to\M$. The kernel of this map is $\P_x\cap \U\subseteq \U$. Thus, the kernel of this map is a unipotent normal subgroup of $\P_x$, and thus it is contained in $\RR_x$.

    Let $w\in \MWM$ be such that $PxP=PwP$. By \ref{Lem 3.6}, there exists $y\in [x]_P\cap M(w)w$ and let $p\in P$ be such that $x=py\bar{p}^{-1}$. We have
    \begin{align*}
        \pr_M(\P_x)&=\pr_M(p)\pr_M(\P_y)\overline{\pr_M(p)^{-1}} &\text{and}\\
        \RR(x)&=p\RR(y)\bar{p}^{-1} &
    \end{align*}

    Moreover, $\P\cap x\P \overline{x}^{-1}=p(\P\cap y\P \bar{y}^{-1})\bar{p}^{-1}$ implies $\pr_M(\P\cap x\P \overline{x}^{-1})=\pr_M(p)\pr_M(\P\cap y\P \bar{y}^{-1})\pr_M(\bar{p}^{-1})$, and thus part (2) follows from \ref{lem 3.7}.

    Part (3) will follow the same argument in \cite[Lemma 3.8]{LO}\label{lem3.8}.
\end{proof}

\begin{lem}\label{lem3.9} 
    Let $x\in N_G(M)\cap X$. We have $\P_x=\M_x\ltimes \U_x$.
\end{lem}

\begin{proof}
    Consider the Levi decomposition $\P\cap x\P x^{-1}=\M\ltimes x\U x^{-1}$, which is invariant under conjugation by $x$. That is, $(\P\cap x\P x^{-1})_x=\M_x\ltimes (\U\cap x\U x^{-1})_x$. Further, $\P_x=(\P\cap x\P x^{-1})_x$ and $\U_x=(\U \cap x\U x^{-1})_x$. Then the lemma follows.
\end{proof}

We have the following lemma from \cite[Lemma 3.10]{LO}.

\begin{lem}\label{lem3.10}
    There exist a bijection between $M$-admissible $P$-orbits in $X$ and $M$-orbits in $N_G(M)\cap X$ by the map $[x]_P\mapsto [x]_P\cap N_G(M)$. 
\end{lem}

\begin{rmk}
    There is a bijection between $\imath_M: N_G(M)/M\to W(M,M)\hookrightarrow W$. Without lost of generality, we will use the same symbol when we precompose this map with the quotient map, that is $\imath_M: N_G(M)\to N_G(M)/M\to W(M,M)$.
\end{rmk}
    
\subsection{Minimal Involutions}\label{secmin.inv}



    


Let $M=\{L_{(n_{1},\cdots,n_{k};m)}\}$ be a standard Levi subgroup of $G$.  
Consider $$[M]=\{L_{(n_{\sigma(1)},\cdots,n_{\sigma(k)};m)}:\sigma\in S_k \}.$$
This is the set of standard Levi subgroups of $G$ which are conjugate to $M$. For $M'\in [M]$, let $W(M,M')=\{w\in W: wMw^{-1}=M'\text{ and }w \text{ is of minimal length in }wW^M\}$. Let $W(M)=\bigcup_{M'\in [M]} W(M,M')$. Notice that if $w\in W(M)$ and $w'\in W(wMw^{-1})$, then $w'w\in W(M)$.

There is a natural set bijection between $W(M)$ and the signed permutation group $\frakW_k$. We first identify the set $\Delta_M$ of simple roots with respect to $M$ to the set $\Delta_k$ of simple roots associate to $\frakW_k$. This identifies the set of elementary symmetries in $W(M)$ (defined in \cite[\S I.1.7]{MW}) and the set of simple reflections in $\frakW_k$. We can define the map $\jmath_M: W(M)\to \frakW_k$ inductively by 
\begin{align*}
    \jmath_M(w'w)&=\jmath_{wMw^{-1}}(w')\jmath_M(w),&
    \begin{split}
        w\in W(M),\;w'\in W(wMw^{-1}).
    \end{split}
\end{align*}
This map is well defined and injective since in $W(M)$ (\textit{resp.} $\frakW_k$), $w$ is uniquely identified by $\{\alpha\in R(T_M,G):\alpha>0,w\alpha<0\}$ (\textit{resp.} $\{\alpha\in R(T_0,U_{2k}):\alpha>0,w\alpha<0\}$). 

\begin{defn}
    An element $w\in W(M,M)$ is called $M$-minimal if $w=w^L_M$ for some standard Levi subgroup $L\supseteq M$ and $w\alpha=-\alpha$ for $\alpha\in \Delta^L_M$.
\end{defn}

A remark is that if $w\in W(M,M)$ is $M$-minimal, then $\jmath_M(w)$ is minimal in $\frakW_k$, the Weyl group of $U_{2k}$. 

Let $w\in W(M,M)$ be $M$-minimal, and let $\jmath_M(w)=\tau\frakc_{l,k}$. We define the following two subsets of $\{1,\cdots,l\}$:
\begin{align*}
    &S=\frakc_-(\jmath_M(w)), &R=\frakc_<(\jmath_M(w)).
\end{align*}

\begin{lem}\label{main lemma}
    With the above notation, $wM \cap X$ is not empty if and only if $n_i$ is even for all $\ell + 1 \leq i \leq k$ and in this case $wM \cap X$ is a unique $M$-orbit.
\end{lem}
\begin{proof}
    Assume that $x \in wM \cap X$. We can choose
\[
t_w = \iota(t_1,\ldots,t_{r+s};t) \in w
\]
where
\[
t_j = \begin{cases}
    I_{n_{i(j)}} & i(j) \in S\\
    \begin{pmatrix}
        0 & I_{n_{i(j)}}\\
        I_{n_{i(j)}} & 0
    \end{pmatrix} & i(j) \in R
\end{cases},
i(1) = 1,
i(j+1) = \begin{cases}
    i(j) + 1 & i(j) \in S\\
    i(j) + 2 & i(j) \in R
\end{cases}
\]
for $j = 1,\ldots,r+s$ (note that $i(j) \in R \cup S$ for all $j$) and
\[
t = 
\begin{pmatrix}
    &&&&&&I_{n_\ell + 1}\\
    &&&&&\cdot&\\   
    &&&&I_{n_k}&&\\
    &&&I_{n_\ell + 1}&&&\\
    &&-I_{n_k}&&&&\\
    &\cdot&&&&&\\
    -I_{n_\ell + 1}&&&&&&
\end{pmatrix}.
\]

Then, by definition, $x \in t_w M \cap X$. Let $ x = t_w \iota (g_1,\ldots,g_k;h)$ with $g_i \in \GL_{n_i}(E)$ and $ h \in G_m$. The condition $x \in X$ is equivalent to
\[
\begin{cases}
    g_i \bar{g}_i = I_{n_i} & i \in S\\
    g_i \bar{g}_{i+1} = I_{n_i} & i \in R\\
    g_i \bar{g}_i^* = - I_{n_i} & \ell + 1 \leq i \leq k\\
    h \bar{h} = I_{2m}.
\end{cases}
\]
In particular, for $\ell + 1 \leq i \leq k$ the condition is that $g_i w_{n_i}$ is an alternating matrix and therefore $n_i$ is even. For $d_i \in \GL_{n_i}(E), i = 1,\ldots,k$ and $d^\prime \in G_m$ let $d = \iota(d_1,\ldots,d_k;d^\prime) \in M$. Then
\[
dx\bar{d}^{-1} = t_w \iota(g_1^\prime,\ldots,g_k^\prime;h^\prime)
\]
where
\[
\begin{cases}
    g_i^\prime = d_i g_i \bar{d}_i^{-1} & i \in S\\
    g_i^\prime = d_{i+1} g_i \bar{d}_i^{-1}, g_{i+1}^\prime = d_i g_{i+1} \bar{d}_{i+1}^{-1}={\bar{g^\prime}}^{-1} & i \in R\\
    g_i^\prime = d_i^* g_i \bar{d}_i^{-1} & \ell + 1 \leq i \leq k\\
    h' = d' h \bar{d'}^{-1}.
\end{cases}
\]
It follows from {Example \ref{two examples of single orbit}} and Lemma \ref{X(U_2n,theta) is single G orbit.} that $t_w M \cap X$ is a unique $M$-orbit.

On the other hand, assuming that $n_i$ is even whenever $\ell + 1 \leq i \leq k$, with the above choice of $t_w$ we have
\begin{align}\label{good xw}
    x_w = t_w \iota(I_{n_1+\ldots+n_\ell}, \epsilon_{\ell+1}, \ldots, \epsilon_{n_k};I_{2m}) \in wM\cap X
\end{align}

where
$\epsilon_{2n} = \begin{pmatrix} -I_n & 0 \\ 0 & I_n \end{pmatrix}$.
\end{proof}

Let $\M$ be a Levi subgroup of $\G$ and $x\in N_G(M)\cap X$. The group $\imath_M(N_G(M)/M)=W(M,M)$ acts on $\fraka_M^*$. In particular, since $x$ and $\bar{x}$ correspond to the same element in $N_G(M)/M$ and also $x\bar x=1$, $x$ acts as an involution on $\fraka_M^*$ and decomposes into a direct sum of eigenspaces with eigenvalues $1$ and $-1$. We will denote $(\fraka_M^*)^+_x$ for the eigenspace for $+1$ and $(\fraka_M^*)^-_x$ for $-1$. Denote $\L=\L(x)$ be the intersection of all \textit{semistandard} Levi subgroups containing $M$ and $x$. We have $(\fraka_M^*)^+_x=\fraka_L^*$ and $(\fraka_M^*)^-_x=(\fraka_M^L)^*$.

\begin{rmk}\cite[Remark 3.17]{LO}
    If $w=\imath_M(x)$, then $L(x)$ and $\fraka_M^*=(\fraka_M^*)^+_x\oplus(\fraka_M^*)^-_x$ depend only in $w$. Furthermore, $L(x)$ is a standard Levi subgroup if and only if $w$ is an $M$-minimal involution and $w=w_M^{L(x)}$.
\end{rmk}

\begin{lem}\label{lem3.19}
For every $x \in N_G(M) \cap X$ the restriction of $H_M$ to $\M_x(\A)$ defines a surjective homomorphism
\[
H_M:\M_x(\A) \rightarrow (\fraka_M)_x^+
\]
Moreover, the restriction of $H_M$ to $A_M^{M_x}$ defines an isomorphism
\[
H_M:A_M^{M_x} \rightarrow (\fraka_M)_x^+.
\]
\end{lem}
\begin{proof}
Since $x e^\nu x^{-1} = e^{x\nu}$ for any $\nu \in \fraka_M$, the second part is true. To show the first part, we need to show $H_M(\M_x(\A)) \subseteq (\fraka_M)_x^+$. This is true since $H_M(x m x^{-1}) = xH_M(m)$ for any $m \in \M(\A)$. 
\end{proof}

\begin{defn}
    Let $x\in N_G(M)\cap X$. We say $x$ is $M$-minimal if $L(x)$ is standard Levi subgroup of $G$.
\end{defn}

\begin{lem}\label{orbitana}
    Let $x\in N_G(M)\cap X$ be $M$-minimal. We have 
    \begin{align*}
        \M_x&\cong \prod_{i\in S} \GLn{n_i}\times \prod_{i\in R} \ResGL{n_i} \times \prod_{i=l+1}^k \ResSp{n_i}\times \Spn{2m}.\\
        \L(x)=\L&\cong\prod_{i\in S} \ResGL{n_i}\times\prod_{i\in R} \ResGL{2n_i}\times \Un{2(n_{l+1}+\cdots+n_{k}+m)}.
    \end{align*}
    Further, if $\L(x)$ is contained in the Siegel parabolic subgroup, then
    \begin{align*}
        \L_x&\cong \prod_{i\in S} \GLn{n_i}\times\prod_{i\in R} \GLn{2n_i}.
    \end{align*}
\end{lem}

\begin{proof}
    Applying Lemma \ref{main lemma} and choose $x=x_w$ (see (\ref{good xw})), we can see that $\M_x$ consists of elements of the form $\imath(g_1,\cdots,g_k;h)$, where
    \begin{align*}
        &g_i\in \GLn{n_i} & \text{ for } i\in S;\\
        &\begin{pmatrix} g_i & 0 \\ 0 & g_{i+1} \end{pmatrix}=\begin{pmatrix} g_i & 0 \\ 0 & \Bar{g_i} \end{pmatrix}\text{ for } g_i\in \ResGL{n_i} & \text{ for } i\in R;\\
        &g_i\in \ResSp{n_i} & \text{ for } i\in \{l+1,\cdots,k\};\\
        &h\in \Spn{2m}.
    \end{align*}
    It is easy to see that $\L(x)\cong\prod_{i\in S} \ResGL{n_i}\times\prod_{i\in R} \ResGL{2n_i}\times \Un{2(n_{l+1}+\cdots+n_{k}+m)}$. When $\L(x)$ is contained in the Siegel parabolic subgroup, i.e. $l=k$ and $m=0$, $\L_x$ consists of elements of the form $\imath(g_1,\cdots,g_l)$, where
    \begin{align*}
        &g_i\in \GLn{n_i} & \text{ for } i\in S;\\
        &g_i=\begin{pmatrix} \alpha & \beta \\ \bar{\beta} & \bar{\alpha} \end{pmatrix}\text{ where } g_i\in \ResGL{2n_i} \text{ and }\alpha,\beta\in \mathbf{Res_{E/F}Mat_{n_i\times n_i}}& \text{ for } i\in R.
    \end{align*}
    Let $\eta_i=I_{n_i}$ for $i\in S$ and $\eta_i = \begin{pmatrix} I_{n_i}&iI_{n_i}\\I_{n_i}&-iI_{n_i} \end{pmatrix}$ for $i \in R$, the isomorphism from $\prod_{i\in S} \GLn{n_i}\times\prod_{i\in R} \GLn{2n_i}$ to $\L_x$ will be given by $(g_1,\cdots,g_l) \mapsto \imath(\eta_1^{-1} g_1 \eta_1,\cdots,\eta_l^{-1} g_l \eta_l)$.
\end{proof}

\begin{defn}
    Let $\M$ be a Levi subgroup of $\G$ and $x\in N_G(M)\cap X$. We say $x$ is $M$-cuspidal if $\L(x)$ is contained in the Siegel parabolic subgroup of $\G$.
\end{defn}

\subsection{Graph}
Let $\P=\M\ltimes\U$ be a parabolic subgroup of $\G$, and let $\alpha\in\Delta_P$. Let $s_\alpha$ be the elementary symmetry associated to $\alpha$. 
We define a directed graph with labeled edge $\frakG$ as follows. The vertices of $\frakG$ are pairs $(M,x)$ where $M$ is a Levi subgroup of $G$ and $x\in N_G(M)\cap X$. The edges are $(M,x)\xrightarrow{n_\alpha}(M',x')$ that satisfy the following conditions:
\begin{enumerate}
    \item $\alpha\in\Delta_P$,
    \item $n_\alpha\in s_\alpha M$,
    \item $x\alpha\neq \pm \alpha$,
    \item $M'=s_\alpha M s_\alpha^{-1}=n_\alpha M n_\alpha^{-1}$,
    \item $x'=n_\alpha x n_\alpha^{-1}$.
\end{enumerate}

Note that if $(M,x)\xrightarrow{n_\alpha}(M',x')$, then $(M',x')\xrightarrow{n_\alpha^{-1}}(M,x)$. 
We denote $(M,x)\overset{n_\alpha}{\searrow}(M',x')$ if $(M,x)\xrightarrow{n_\alpha}(M',x')$ and $x\alpha<0$. Also, either $(M,x)\overset{n_\alpha}{\searrow}(M',x')$ or $(M',x')\overset{n_\alpha^{-1}}{\searrow}(M,x)$.

For a finite sequence of edges
\begin{align*}
    (M_1,x_1)\xrightarrow{n_{\alpha_1}}(M_2,x_2)\xrightarrow{n_{\alpha_2}}\cdots\xrightarrow{n_{\alpha_k}}(M_{k+1},x_{k+1})
\end{align*}
we write $(M_1,x_1)\overset{n}{\curvearrowright}(M_{k+1},x_{k+1})$ where $n=n_{\alpha_k}\cdots \alpha_1$. For a finite sequence of edges
\begin{align*}
    (M_1,x_1)\overset{n_{\alpha_{1}}}{\searrow}(M_2,x_2)\overset{n_{\alpha_{2}}}{\searrow}\cdots\overset{n_{\alpha_{k}}}{\searrow}(M_{k+1},x_{k+1})
\end{align*}
we write $(M_1,x_1)\overset{n}{\downarrow}(M_{k+1},x_{k+1})$ where $n=n_{\alpha_k}\cdots \alpha_1$.

\begin{lem}\label{graph}
    Given two vertices $(M,x)$ and $(M',x')$ and edge $(M,x)\overset{n_\alpha}{\searrow}(M',x')$ for some $\alpha\in\Delta_P$. Let $\Q=\L\ltimes\V$ be a parabolic subgroup of $\G$ such that $\Delta_P^Q=\{\alpha\}$. Let $\P'=\M'\ltimes\U'$ be a parabolic subgroup of $\G$ such that $\Delta_{P'}^Q=\{-s_\alpha(\alpha)\}$. Then we have the following.
    \begin{enumerate}
        \item $\V_x=n_\alpha \U_x n_\alpha^{-1}$, and further $n_\alpha \U_x n_\alpha^{-1}\subseteq\U_{x'}'$.
        \item We have the following short exact sequence of subgroups:
        \[
        \begin{tikzcd}
              1 \ar[r] & n_\alpha \U_x n_\alpha^{-1}  \ar[r]& \U_{x'}'  \ar[r, "\pr_L"] & \L\cap\U'  \ar[r] & 1. 
        \end{tikzcd}
        \]
        \item For $f: \V\backslash \U'(\A)\to \C$, we have 
        \begin{align*}
            \int_{n_\alpha\U_x(\A) n_\alpha^{-1}\backslash\U_{x'}'(\A)} f(u)\;du=\int_{\V(\A)\backslash\U'(\A)}f(u)\;du=\int_{(\L\cap\U')(\A)}f(u)\;du
        \end{align*}
        whenever the integral is defined.
        \item $n_\alpha\P_x n_\alpha^{-1}\subseteq\P_{x'}'$, and the semi-invariant measure on $n_\alpha\P_x(\A) n_\alpha^{-1}\backslash\P_{x'}'(\A)$ is given by the integration over $n_\alpha\U_x(\A) n_\alpha^{-1}\backslash\U_{x'}'(\A)$.
        \item We have the following equations.
        \begin{align*}
            &\delta_{P_x}(m)=(\delta_{P'_{x'}}\delta^{-1}_{P'\cap L})(n_\alpha m n_\alpha^{-1}), &\text{for all }m\in\M_x(\A),\\
            &(\delta_P^{-\frac{1}{2}}\delta_{P_x})(m)=(\delta_{P'_{x'}}\delta^{-\frac{1}{2}}_{P'})(n_\alpha m n_\alpha^{-1}), &\text{for all }m\in\M_x(\A).
        \end{align*}
        In particular, $$n_\alpha \rho_x=\rho_{x'}.$$
    \end{enumerate}
\end{lem}

\begin{proof}
    The lemma is proved in \cite[Lemma 4.3.1]{LR}.
\end{proof}

We have the following corollary.

\begin{cor}\label{Straightforward equation of unimodular character}
    Given $(M,x)\overset{n}{\curvearrowright}(M',x')$, we have
    \begin{align*}
            &(\delta_P^{-\frac{1}{2}}\delta_{P_x})(m)=(\delta_{P'_{x'}}\delta^{-\frac{1}{2}}_{P'})(n_\alpha m n_\alpha^{-1}), &\text{for all }m\in\M_x(\A).
    \end{align*}
    In particular, $$n_\alpha \rho_x=\rho_{x'}.$$
\end{cor}

Moreover, by Corollary \ref{Straightforward equation of unimodular character} and \cite[Lemma 3.2.1 and Proposition 3.3.1]{LR}, we have the following.

\begin{cor}\label{graph.fin}
    Let $M$ be a Levi subgroup of $G$ and $x\in N_G(M)\cap X$. Then there exists $n\in G$ such that $M'=nMn^{-1}$ is standard Levi subgroup, $x'=nxn^{-1}$ is $M'$-minimal, and $(M,x)\overset{n}{\downarrow}(M',x')$. Let $P'$ be the standard parabolic subgroup, with $M'$ as its Levi part. Then, we have
    \begin{align*}
            &(\delta_P^{-\frac{1}{2}}\delta_{P_x})(m)=(\delta_{P'_{x'}}\delta^{-\frac{1}{2}}_{P'})(n_\alpha m n_\alpha^{-1}), &\text{for all }m\in\M_x(\A).
    \end{align*}
    In particular, $$n_\alpha \rho_x=\rho_{x'}.$$.
\end{cor}

\section{Period Integral}
\subsection{Vanishing Pairs}
In this section let $\G$ be a reductive group and $\H$ a reductive subgroup defined over $F$. Recall the notation
\[
[\H]_G = A_G^H H \mo \H(\A) \subseteq [\G].
\]

\begin{defn}
A pair $(\G,\H)$ of groups is a vanishing pair if
\[
\int_{H \mo \H(\A)^1} \phi(h) dh = 0
\]
for every smooth cuspidal function of uniform moderate growth $\phi$ on $[\G]$.
\end{defn}
We have the following vanishing pairs.
\begin{thm}
The following are vanishing pairs
\begin{enumerate}
    \item $(\U_{2n},\Sp_{2n})$ (see \cite{MO2});
    \item $(\GL_{2n},\Sp_{2n})$ (see \cite{JR}).
\end{enumerate}
\end{thm}

\begin{cor}\label{cor4.5}
Let $(\G,\H) = (\U_{2n},\Sp_{2n})$, and let $\M$ be a Levi subgroup of $\G$. Then for any $x \in N_G(M) \cap X$, $(\M,\M_x)$ is a vanishing pair unless $x$ is $M$-cuspidal.
\end{cor}
\begin{proof}
The result follows from the orbit analysis in the pervious chapter. The central block of the pair is conjugate to $(\U_{2m},\Sp_{2m})$. Hence the corresponding period integral admits a zero factor.
\end{proof}
\subsection{The Intertwining Period}
\subsubsection{Definition}
Let $\P =\M \ltimes \U$ be a parabolic subgroup of $\G$, and let $x \in N_G(M) \cap X$. For $\varphi \in \Au_P^{mg}(G)$ and $\lambda \in \rho_x + (\fraka_{M,\C}^*)_x^-$, we define, whenever convergent,
\[
J(\varphi,x,\lambda) = \int_{A_M^{M_x}U_x(\A)M_x \mo G_x(\A)} \varphi_\lambda(h \eta_x) dh,
\]
where $\eta_x \in G$ is such that $\eta \cdot e = x$.
Note that the integral formally makes sense and does not depend on the choice of $\eta$.
Moreover, we have
\begin{align*}
 J(\varphi,x,\lambda) &= \int_{\P_x(\A) \mo \G_x(\A)} \int_{[\M_x]_M} \delta_{P_x}^{-1}(m) \varphi_\lambda(mh\eta) dm dh\\
 &= \int_{\P_x(\A) \mo \G_x(\A)} e^{\langle \lambda, H_M(h\eta) \rangle} \int_{[\M_x]_M} \delta_{P_x}^{-1}(m) e^{\langle \rho_x, H_M(m) \rangle} \varphi(mh\eta) dm dh,
\end{align*}
by noticing that $H_M(mh\eta) = H_M(m) + H_M(h\eta)$ and $\lambda - \rho_x \in (\fraka_{M,\C}^*)_x^-$.

\vskip 1cm 

\subsubsection{Convergence of the Intertwining Period}
Let $\Sigma_{P,x} = \{\alpha \in \Sigma_P \mid x\alpha < 0\}$. For $\gamma > 0$, we define
\[
\mathfrak{D}_x(\gamma) = \rho_x + \{\lambda \in (\fraka_M^*)_x^-: \langle\lambda, \alpha^\vee\rangle > \gamma, \forall\alpha \in \Sigma_{P,x}\}.
\]
When $x$ is $M$-minimal, recall that $w= \imath(x) \in W(M,M)$ is $M$-minimal and $L = L(x) = L(w)$ is a standard Levi subgroup. Furthermore, we have:
\[
(\fraka_{M,\C}^*)_x^- = (\fraka_M^L)_\C^* = \{(\lambda_i)_i \in \C^k \mid \lambda_i = -\lambda_i \text{ for } i \in R \text{ and } \lambda_i = 0 \text{ otherwise}\},
\]
where $R = \setc_<(\mathfrak{J}_M(w))$ is defined in Section 3.4.

We have the following result (cf. \cite{LR},\cite{LO}).
\begin{lem}
Let $(M,x)$ and $(M^\prime,x^\prime)$ be vertices in the graph $\mathfrak{G}$ such that $(M,x) \overset{n_\alpha}{\searrow} (M^\prime,x^\prime)$ be vertices for some $\alpha \in \Delta_P$ and $n_\alpha \in s_\alpha M$. Then
\[
\mathfrak{D}_x(\gamma) = s_\alpha^{-1} \mathfrak{D}_{x^\prime}(\gamma) \cap (\rho_x + \{\lambda \in (\fraka_M^*)_x^-: \pair{\lambda}{\alpha^\vee} > \gamma\}).
\]
\end{lem}

Let $\M$ be a Levi subgroup of $\G$ contained in the Siegel parabolic subgroup. We have the following theorem on the convergence of the intertwining period.

\begin{thm}\label{convthm}
There exists $\gamma > 0$, for all $x \in N_G(M) \cap X$ and $\varphi \in \mathcal{A}_P^{rd}(G)$ the integral defining $J(\varphi,x,\lambda)$ is absolutely convergent for $\mathrm{Re}\;\lambda \in \mathfrak{D}_x(\gamma)$. Moreover, for any compact subset $D$ of $\mathfrak{D}_x(\gamma)$ there exists $N,C>0$, for all $\lambda \in D + i(\fraka_M^*)_x^-$
\[
\int_{A_M^{M_x}U_x(\A)M_x \mo G_x(\A)} |\varphi_\lambda(h\eta)|dh \leq C \sup_{m \in \mathfrak{S}_M^1 \atop k \in K} |\varphi(mk)|\|m\|^N.
\]
\end{thm}

Indeed, Theorem \ref{convthm} is implied by the following proposition. We will prove the implication in this section and prove the proposition in the following sections.

\begin{prop}\label{convprop}
There exist $R$ and $\gamma$, for all $x \in N_G(M) \cap X$, the integral
\[
J(\theta_f^M,x,\lambda) = \int_{A_M^{M_x}U_x(\A)M_x \mo G_x(\A)}(\theta_f^M)_\lambda (h\eta) dh
\]
is absolutely uniformly convergent on $D+i(\fraka_M^*)_x^-$ for any compact $D$ in $\mathfrak{D}_x(\gamma)$. 
\end{prop}
In order to see this, we first introduce the following lemmas.
\begin{lem}\label{lowerbound}
For all $R>0$, there exists $N>0$, s.t.
\[
\|m\|^{-N} \ll \theta_f^M(mk),\ m\in\mathfrak{S}_M^1,k\in K,
\]
where $f = e^{-R\|\cdot\|} \in C_R(\fraka_0^M)$.
\end{lem}
\begin{proof}
By the right $K$-invariance
\[
\theta_f^M(mk) = \theta_f^M(m)=\sum_{\gamma \in P_0 \cap M \mo M} e^{\langle \rho_0,H_0(\gamma m)\rangle}f(H_0^M(\gamma m)).
\]
Note that the series is positive. By taking only the term when $\gamma = e$ in the series
\[
\theta_f^M(mk) \geqslant e^{\langle \rho_0,H_0(m)\rangle}f(H_0^M(m)).
\]
By (\ref{5c}),
\[
\theta_f^M(mk) \geqslant e^{(-\|\rho_0\|-R)\|H_0(m)\|} \gg e^{(-\|\rho_0\|-R)(1+\log\|m\|)}.
\]
By (\ref{5d}),
\[
\theta_f^M(mk) \gg \|m\|^{-N}
\]
for $N = \|\rho_0\|+R$.
\end{proof}

\begin{lem} 
For any $R>0$ and $\varphi \in \mathcal{A}_P^{rd}(G)$, there exists $N$, for all $g \in G(\A)$
\[
\varphi(g) \ll \sup_{m\in \mathfrak{S}_M^1 \atop k \in K} |\varphi(mk)|||m||^N|\theta_f^M(g)|,
\]
where $f = e^{-R\|\cdot\|}$.
\end{lem}
\begin{proof}

Write $g = uamk$ where $u \in U(\A), a \in A_M, m \in M(\A)^1, k\in K$ by Iwazawa decomposition. Note that if a function $\varphi$ on $\U(\A)M \mo \G(\A)$ satisfies $\varphi(ag) = e^{\langle\rho_P,H_0(a)\rangle}\varphi(g)$, we have
\begin{equation}\label{twist}
\varphi(g) = e^{\langle\rho_P,H_0(a)\rangle}\varphi(mk).
\end{equation}

Apply (\ref{twist}) to $\varphi$,
\[
|\varphi(g)| = e^{\langle\rho_P,H_0(a)\rangle}|\varphi(mk)|.
\]

Let $N$ be as in Lemma \ref{lowerbound},
\[
|\varphi(g)| \ll e^{\langle\rho_P,H_0(a)\rangle}|\varphi(mk)|\|m\|^N |\theta_f^M(mk)|.
\]

Apply (\ref{twist}) to $\theta_f^M$. The right hand side of the above formula equals to $|\varphi(mk)|\|m\|^N |\theta_f^M(g)|$.
By the $M$-invariance, this is bounded by the supremum when $m$ runs over the Siegel domain $\mathfrak{S}_M^1$. Hence
\[
|\varphi(g)| \leq \sup_{m\in M(\A)^1 \atop k \in K}|\varphi(mk)|\|m\|^N|\theta_f^M(g)|.
\]
\end{proof}

For any $\lambda \in \fraka_M^*$, we also have
\[
|\varphi_\lambda(g)| \ll \sup_{m\in \mathfrak{S}_M^1 \atop k \in K}|\varphi(mk)|\|m\|^N|(\theta_f^M)_\lambda(g)|.
\]

Now we can explain the reduction to the proposition. Let $M$ and $x$ be as above, and $\varphi \in \Au_P^{rd}(G)$. Let $\gamma$ and $R$ be as in Proposition \ref{convprop}, then for all $\lambda \in \mathfrak{D}_x(\gamma) + i(\fraka_M^*)_x^-$:
\begin{align*}
\int_{A_M^{M_x}\U_x(\A)M_x \mo \G_x(\A)} |\varphi_\lambda(h\eta)|dh
&\ll \int_{A_M^{M_x}\U_x(\A)M_x \mo \G_x(\A)}  \sup_{m \in \mathfrak{S}_M^1 \atop k \in K} |\varphi(mk)|\|m\|^N |(\theta_f^M)_\lambda(h\eta)|dh\\
&\ll \int_{A_M^{M_x}\U_x(\A)M_x \mo \G_x(\A)}|(\theta_f^M)_\lambda(h\eta)|dh,
\end{align*}
where the last integral is convergent.

\subsection{Proof of the Convergence of the Intertwining Period}
\subsubsection{Lemmas}
In this sections, we introduce some lemmas related to the proof of the convergence of the intertwining integral.




Recall the function
\[
\theta_f^M(g) = \sum_{\gamma \in P_0 \cap M \mo M} e^{\pair{\rho_0}{H_0(\gamma g)}} f(H_0^M(\gamma g)), g \in \G(\A)
\]
defined in Section 2.6.
We have the following property of the function.
\begin{lem}
There exists $R>0$, such that the function $\theta_f^M$ is convergent and bounded on $\M(\A)^1$.
\end{lem}
\begin{proof}
Note that $\|g\|$ has a lower bound when $g$ runs over the whole group $\G(\A)$, according to (\ref{5a}). Hence
\[
\sup_{m \in \mathfrak{S}_M^1} |\theta_f^M(m)| \leq C \sup_{m \in \mathfrak{S}_M^1} |\theta_f^M(m)|\|m\|.
\]
By Lemma \ref{lem2.2}, the above is bounded by
\[
C\|1\|^{N^\prime} = C.
\]
where $N^\prime$ is some integer coming from the lemma.
\end{proof}

\begin{lem}\label{omer}
Let $f$ be such that $\theta_f^M$ convergent. We have the following inequality:
\[
\sup_{m \in \mathfrak{S}_M^1}|\theta(mg)| \ll e^{\langle \rho_P,H_0(g) \rangle} \text{ for any } g \in \G(\A).
\]
\end{lem}
\begin{proof}
Recall from Section 2.6
that $\theta_f^M$ has the following properties:
\begin{enumerate}
    \item[(i)] $\theta_f^M(ag) = e^{\langle \rho_P, H_0(a) \rangle} \theta_f^M(g)$ for any $a \in A_M, g\in \G(\A)$;
    \item[(ii)] $\theta_f^M$ is right-$K$ invariant.
\end{enumerate}

Let $g = uam_1k$ where $u \in \U(\A)$, $a \in A_M$, $m_1 \in \M(\A)^1$ and $k \in K$. Then
\begin{align*}
 \theta(mg) = \theta(muam_1k) = \theta(au^\prime m m_1 k) = e^{\langle \rho_P,H_0(a) \rangle} \theta(mm_1),
\end{align*}
where $u^\prime$ is such that $u^\prime m = mu$.
Note that $H_0(g) = H_0(a)$. By the previous lemma, we have
\begin{align*}
|\theta(mg)| \ll e^{\langle \rho_P,H_0(a) \rangle}\\
= e^{\langle \rho_P,H_0(g) \rangle},
\end{align*}
for all $m \in \M(\A)^1$ and $g \in \G(\A)$. The lemma is shown by having $m$ run over $\mathfrak{S}_M^1$.
\end{proof}



In the rest of this section we let $\G = \Res_{E/F}\GL_{2n}$ and $\M = \Res_{E/F} \GL_n \times \Res_{E/F}\GL_n$. We let the maximal compact group $K$ of $\G(\A)$ be the standard one. The map $H_M$ is defined as that in Section 2.2.
We further let 
$$\eta = \begin{pmatrix} I_n&iI_n\\I_n&-iI_n \end{pmatrix}$$ and 
$$x = \eta \cdot e = \begin{pmatrix} 0&I_n\\I_n&0 \end{pmatrix}.$$
We consider the pair $(\L_x,\M_x)$ and have the following lemma.
\begin{lem}\label{lemma27}
There exists $t_0$, for all compact subset $K_1$ in $(t_0, \infty)$, 
\[
\int_{\M_x(\A)\mo \L_x(\A)} e^{\langle(t,-t), H_M(l) \rangle} dl \ll_{K_1} 1, \text{ for all $l \in \L_x(\A)$}.
\]
\end{lem}
\begin{proof}
Apply the conjugation of $\eta$ to the variable in the integral:
\begin{align*}
    \int_{\M_x(\A)\mo \L_x(\A)} e^{\langle(t,-t), H_M(l) \rangle} dl &= \int_{\M_x(\A)\mo \L_x(\A)} e^{\langle(t,-t), H_M(l\eta) \rangle} dl\\
 &=\int_{\eta^{-1}\M_x(\A)\eta \mo \eta^{-1}\L_x(\A)\eta} e^{\langle(t,-t), H_M(\eta l) \rangle} dl.
 \end{align*}
The convergence follows from Lemma 27 in \cite{JLR}.
\end{proof}

\vskip 1cm 

\subsubsection{The Minimal Case}
\indent As we mentioned, it suffices to show Theorem \ref{convthm} by proving Proposition \ref{convprop}.
In the rest of this section, we will prove Proposition \ref{convprop} for the case when $x$ is $M$-minimal. It suffices to show that for $f = e^{-R\|\cdot\|}$ and $\lambda$ real. The function $(\theta_f^M)_\lambda$ is positive in this case. We will first show for the case when $x$ is $M$-minimal and then reduce all the cases to this case.

Let $x$ be $M$-minimal and $\Q = \L \ltimes \V$ be the standard parabolic subgroup containing $\L = \L(x)$. We write the intertwining period $J(\theta_f^M,x,\lambda)$ as
\begin{align*}
\int_{\Q_x(\A)\mo \G_x(\A)} &\int_{\P_x(\A)\mo \Q_x(\A)} \delta_{Q_x}^{-1}(q) e^{\langle \lambda,H_P(qh\eta) \rangle}\\
&\times \int_{[\M_x]_M} \delta_{P_x}^{-1}(m) e^{\langle \rho_x,H_P(m) \rangle} \theta_f^M(mqh\eta) dm~dq~dh.
\end{align*}
Since $\Q_x$ is a parabolic subgroup of $\G_x$, we have $\Q_x(\A) \mo \G_x(\A)$ is compact. Let $K_0$ be a compact subset of $\G_x(\A)$ such that $\Q_x(\A) K_0 = \G_x(\A)$. It suffices to show that there exist $R>0$ and $\gamma>0$, for all compact $D \subset \mathfrak{D}_x(\gamma)$ and compact $K_0 \subset \G_x(\A)$, for all $\lambda$ with $\mathrm{Re} \lambda \in D$ and $h \in K_0$, and for $f = e^{-R\|\cdot\|}$:
\begin{align*}
\int_{\P_x(\A)\mo \Q_x(\A)} &\delta_{Q_x}^{-1}(q) e^{\langle \lambda,H_P(qh\eta) \rangle}\\
&\times \int_{[\M_x]_M} \delta_{P_x}^{-1}(m) e^{\langle \rho_x,H_P(m) \rangle} \theta_f^M(mqh\eta) dm~dq \ll_{D,K_0} 1.
\end{align*}
According to the isomorphism $\M_x(\A)\mo \L_x(\A) \overset{\sim}{\rightarrow} \P_x(\A)\mo \Q_x(\A)$ induced by the inclusion, we take representatives in $\L_x(\A)$ in the above integral. Hence we can transform the above integral as:
\begin{align}
&\int_{\M_x(\A)\mo \L_x(\A)} e^{\langle \lambda - \rho_x ,H_P(lh\eta) \rangle} \label{int1}\\
&\times \delta_{Q_x}^{-1}(l) e^{\langle \rho_x,H_P(lh\eta) \rangle} \int_{[\M_x]_M} \delta_{P_x}^{-1}(m) e^{\langle \rho_x,H_P(m) \rangle} \theta_f^M(mlh\eta) dm~dl. \label{int2}
\end{align}
In order that the above outer integral makes sense, the integrand must be left $\M_x(\A)$-invariant. Moreover, we note that the integrand of the outer integral in both rows (\ref{int1}) and (\ref{int2}) are left $\M_x(\A)$-invariant. Moreover, the integrand of the outer integral in the second row (\ref{int2}) admits a $\M_x(\A)$-invarint upper bound. We see this by showing the following lemma.
\begin{lem}\label{lemint}
There exists $R>0$, for any compact subset $K_0 \subset \G(\A)$ and any $h \in K_0$:
\[
\int_{[\M_x]_M} \delta_{P_x}^{-1}(m) e^{\langle \rho_x,H_P(m) \rangle} \theta_f^M(mlh\eta) dm \ll_{K_0} e^{\langle \rho_P,H_M(l) \rangle}.
\]
\end{lem}
\begin{proof}
According to Lemma \ref{orbitana}, the pair $(\M, \M_x)$ is a product of pairs of the following types:
\begin{enumerate}
 \item[(i)] $(\mathbf{Res_{E/F}} \GL_n, \GL_n)$;
 \item[(ii)] $(\mathbf{Res_{E/F}} \GL_n \times \mathbf{Res_{E/F}} \GL_n, \mathbf{Res_{E/F}} \GL_n)$.
\end{enumerate}
In particular, we have $A_M^{M_x} = A_{M_x}$ and hence $[\M_x]_M$ is of finite volume. Note that $[\M_x]_M$ is isomorphic to $M_x\mo \M_x(\A)^1$. Consequently, we have the upper bound of the inner integral:
\begin{align*}
\int_{M_x\mo \M_x(\A)^1} \delta_{P_x}^{-1}(m) e^{\langle \rho_x,H_P(m) \rangle} \theta_f^M(mlh\eta) dm
=& \int_{M_x\mo \M_x(\A)^1} \theta_f^M(mlh\eta) dm\\
\ll & \sup_{m \in \mathfrak{S}_M^1} \theta_f^M(mlh\eta).
\end{align*}
Further apply Lemma \ref{omer} to $\theta_f^M$:
\begin{align*}
\sup_{m \in \mathfrak{S}_M^1} \theta_f^M(mlh\eta)
\leq  e^{\langle \rho_P, H_M(lh\eta) \rangle} \ll_{K_0}  e^{\langle \rho_P, H_M(l) \rangle}.
\end{align*}
\end{proof}
With the lemma, it is immediate that the integrand in (\ref{int2}) is bounded by
\[
\delta_{Q_x}^{-1}(l)
e^{\langle \rho_P+\rho_x,H_M(l) \rangle}
\]
which is $\M_x(\A)$-invariant.

Since $h$ runs over a compact set, the integrand $e^{\langle \lambda - \rho_x, H_P(lh\eta) \rangle}$ in (\ref{int1}) is bounded by a constant multiple of $e^{\langle \lambda - \rho_x, H_P(l)\rangle}$. It is only left to show that there exists $\gamma>0$, for any compact $D \subset \mathfrak{D}_x(\gamma)$, for any $\lambda \in \rho_x + (\fraka_{M,\C}^*)_x^-$ where $\mathrm{Re} \lambda \in D$:
\begin{equation}\label{LxMx}
\int_{\M_x(\A) \mo \L_x(\A)} \delta_{Q_x}^{-1}(l) e^{\langle \rho_P + \rho_x, H_M(l) \rangle} e^{\langle \lambda - \rho_x, H_M(l) \rangle} dl \ll_D 1.
\end{equation}
Recall the index sets $S$ and $R$ in Section 3.4
and note that $(\M_x(\A),\L_x(\A))$ is a product, running over the index $i$, of pairs of the following two types under the usual setting:
\begin{enumerate}
    \item When $i \in R$,
    \[
    \M_{i,x}(\A) =
   \left\{
\begin{pmatrix} \alpha & 0 \\ 0 & \bar{\alpha} \end{pmatrix} \mid \alpha \in \GL_{n_i}(\A_E)
\right\}
\]
and
\[
\L_{i,x}(\A) = 
        \left\{
\begin{pmatrix} \alpha & \beta \\\bar{\beta} & \bar{\alpha} \end{pmatrix} \in \GL_{2n_i}(\A_E) \mid \alpha,\beta \in \mathbf{Mat}_{n_i\times n_i}(\A_E)
\right\},
\]
   \item When $i \in S$,
   \[
   \M_{i,x}(\A) = \L_{i,x}(\A) = \GL_{n_i}(\A).
   \]
\end{enumerate}
The integral (\ref{LxMx}) decomposes into a product of integrals over the factors and hence it suffices to show the inequality for each of the factors $(\M_{i,x},\L_{i,x})$. As the second type is trivial, it suffices to show for each $i \in R$ of the following inequality:
\begin{equation}\label{LixMix}
\int_{\M_{i,x}(\A) \mo \L_{i,x}(\A)} \delta_{Q_x}^{-1}(l_i) e^{\langle \rho_P + \rho_x, H_M(l_i) \rangle} e^{\langle \lambda - \rho_x, H_M(l_i) \rangle} dl_i \ll_D 1,
\end{equation}
where we abused $l_i$ and $(0,\ldots,0,l_i,0,\ldots,0)$ in the above .

Note that the map $l \mapsto \eta^{-1} l \eta$ gives an isomorphism from $L_{i,x}(\A)$ to $\GL_{2n_i}(\A)$, where $\eta = \begin{pmatrix} I & iI\\ I & -iI \end{pmatrix}$. Hence 
\[
\delta_{Q_x}(l_i) = |\det(l_i)|^{\sigma_i} = e^{\langle \mu_i, H_M(l_i) \rangle}
\]
for some $\mu_i \in \fraka_M^*$. Indeed $\mu_i = (0,\ldots,0,\sigma_i,\sigma_i,0,\ldots,0) \in \fraka_M$. 
Since
$$\delta_{Q_x}(l) e^{\langle \rho_P+\rho_x, H_M(l) \rangle} = e^{\langle \mu_i+\rho_P+\rho_x, H_M(l) \rangle}$$ 
is $M_x(\A)$-invariant, we have that $\mathrm{pr}_{M_i\times M_{i+1}}(\mu_i + \rho_P + \rho_x)$ lies in $\fraka_{M_i}^{L_i}$. By Lemma \ref{lemma27},
the inequality (\ref{LixMix}) holds when $\langle \lambda, \alpha_i^\vee \rangle > \gamma_i$. Let $\gamma = \max_{i \in R}(\gamma_i)$. The convergence holds with $R$ and $\gamma$ for the $M$-minimal case.


\vskip 1cm 

\subsubsection{The General Case}
In order to prove for the general case when $x$ is not supposed to be $M$-minimal, it suffices to prove the following lemma.
\begin{lem}\label{reduction}
Suppose that $(M,x)$ and $(M^\prime,x^\prime)$ are vertices in $\mathfrak{G}$ and $(M,x) \overset{n_\alpha}{\searrow} (M^\prime,x^\prime)$ for some $\alpha \in \Delta_P$. Assume that  Proposition \ref{convprop} holds for $(M^\prime,x^\prime)$. Then it also holds for $(M,x)$. Moreover, there exists $\gamma > 0$ and $R>0$ such that for $\mathrm{Re} \lambda \in \mathfrak{D}_x(\gamma)$ and $f \in C_R(\fraka_0^M)$ we have
\[
J(\theta_f^M,x,\lambda) = J(M(s_\alpha,\lambda)\theta_f^M,x^\prime,s_\alpha \lambda).
\]
\end{lem}
Since every vertex $(M,x)$ reduces to a minimal pair $(M^\prime,x^\prime)$ in finite steps by Theorem \ref{graph.fin}, the minimal case and the above lemma imply the general case.

To show the lemma, assume for Proposition \ref{convprop} holds for $(M^\prime, x^\prime)$ with $R^\prime>0$ and $\gamma^\prime>0$. Let $(M,x)$ be a vertex such that $(M,x) \overset{n_\alpha}{\searrow} (M^\prime,x^\prime)$.

We first show that there exists $R,\gamma$, for any $f = e^{-R\sqrt{1+\|\cdot\|^2}} \in C_R(\fraka_0^M)$, for any $\lambda \in \mathfrak{D}_x(\gamma)$, $J(M(s_\alpha,\lambda)\theta_f^M,x^\prime,s_\alpha\lambda)$ is convergent. We recall the following formula, which can be derived via local computation.

\begin{lem}[Gindikin-Karpelevic Formula]
Let $f = e^{-R\sqrt{1+\|\cdot\|^2}}$ where $\theta_f^M$ is convergent. For any $\alpha \in \Delta_P$, and $\gamma$ large enough so that $\mathrm{Re}\langle \lambda + \mu,\beta^\vee \rangle$ large enough for all $\beta \in \Sigma_{B,s_\alpha}$ so that $M(s_\alpha,\lambda)$ is defined, we have
\[
M(s_\alpha,\lambda)\theta_f^M = \theta_{f^\prime}^{M^\prime},
\]
where $f^\prime$ is the function on $\fraka_0^{M^\prime}$ such that $\hat{f}^\prime(s_\alpha \mu) = c_{s_\alpha}(\lambda + \mu)\hat{f}(\mu)$, and
\[
c_w(\nu) = \prod_{\beta \in \Sigma_{B,w}} \frac{\zeta_F^*(\langle \nu,\beta^\vee\rangle)}{\zeta_F^*(\langle \nu,\beta^\vee\rangle+1)},
\]
where $\zeta_F^*(s)$ is the completed Dedekind zeta function of $F$.
\end{lem}

Note that By Lemma \ref{pw}, for any $R_0 < R$ there exists $\gamma_0$, such that $f^\prime \in C_{R_0}(\fraka_0^{M^\prime})$ for all $\lambda \in \mathfrak{D}_x(\gamma_0)$. Let $R$ be such that $R > R^\prime$, by the above argument, for $R^\prime < R$, we can let $\gamma = max(\gamma_0,\gamma^\prime)$ so that $f^\prime \in C_{R_0}(\fraka_0^M)$ for all $\mathrm{Re}\lambda \in \mathfrak{D}_x(\gamma)$. Under the choice of $R$ and $\gamma$,
\[
J(M(s_\alpha,\lambda)\theta_f^M,x^\prime,s_\alpha\lambda) =J(\theta_{f^\prime}^{M^\prime},x^\prime,s_\alpha\lambda)
\]
is convergent, since $f^\prime \in C_{R^\prime}(\fraka_0^{M^\prime})$ and
$\mathrm{Re}s_\alpha\lambda \in \mathfrak{D}_{x^\prime}(\gamma^\prime)$.

We now prove that
\[
J(M(s_\alpha,\lambda)\theta_f^M,x^\prime,s_\alpha\lambda) = J(\theta_f^M,x,\lambda).
\]
Let $\Q = \L \ltimes \V$ be the parabolic subgroup of $\G$ containing $\P$ such that $\Delta_P^Q = \{\alpha\}$. Let $\eta^\prime = n_\alpha \eta$ and note that $U^\prime \cap s_\alpha U s_\alpha^{-1} = V$. We have
\begin{equation*}
J(M(s_\alpha,\lambda)\theta_f^M,x^\prime,s_\alpha\lambda) = \int_{\P^\prime_{x^\prime}(\A) \mo \G_{x^\prime}(\A)}\int_{[\M^\prime_{x^\prime}]_{M^\prime}}\int_{\V(\A) \mo \U^\prime(\A)} (\theta_f^M)_\lambda (n_\alpha^{-1} umh\eta^\prime)du \delta_{P^\prime_{x^\prime}}(m)^{-1} dmdh.
\end{equation*}
Since the triple integral is convergent and the integrand is nonnegative, it is absolutely convergent. By a change of variable $u \mapsto mum^{-1}$ and reorder the integrals, we have
\[
\int_{\P^\prime_{x^\prime}(\A) \mo \G_{x^\prime}(\A)}
\int_{\V(\A) \mo \U^\prime(\A)}
\int_{[\M^\prime_{x^\prime}]_{M^\prime}} (\theta_f^M)_\lambda (n_\alpha^{-1} muh\eta^\prime)\delta_{P^\prime_{x^\prime}}(n_\alpha mn_\alpha^{-1})^{-1} dmdudh.
\]
By Lemma \ref{graph} (3) and (4), the above integral equals
\begin{align*}
   & \int_{\P^\prime_{x^\prime}(\A) \mo \G_{x^\prime}(\A)}
\int_{n_\alpha\U_x n_\alpha^{-1}(\A) \mo \U^\prime_{x^\prime}(\A)}
\int_{[\M^\prime_{x^\prime}]_{M^\prime}} (\theta_f^M)_\lambda (n_\alpha^{-1} muh\eta^\prime)\delta_{P^\prime_{x^\prime}}(n_\alpha mn_\alpha^{-1})^{-1} dmdudh\\
=&\int_{(n_\alpha \P^\prime_{x^\prime} n_\alpha^{-1})(\A) \mo \G_{x^\prime}(\A)}
\int_{[\M^\prime_{x^\prime}]_{M^\prime}} (\theta_f^M)_\lambda (n_\alpha^{-1} mh\eta^\prime)\delta_{P^\prime_{x^\prime}}(n_\alpha mn_\alpha^{-1})^{-1} dmdudh.
\end{align*}
By the changes of variables $h \mapsto n_\alpha h n_\alpha^{-1}$ and $m \mapsto n_\alpha m n_\alpha^{-1}$, this is exactly $J(\theta_f^M,x,\lambda)$.

To obtain the proof for all $f \in C_R(\fraka_0^M)$, we simply note that every $f \in C_R(\fraka_0^M)$ is bounded by a constant multiple of $e^{-R\sqrt{1+\|\cdot\|^2}}$.

\subsection{Period Integrals of Pseudo-Eisenstein Series}
Let $\P = \M \ltimes \U$ be a standard parabolic subgroup of $\G$.

\begin{thm}\label{main theorem}
There exists $R>0$ such that the period integral
\[
\int_{[H]} \theta_\phi (h) dh
\]
converges absolutely for any $\phi \in C_R(\U(\A)M \mo \G(\A))$ and vanishes when $\M$ is not contained in the Siegel parabolic subgroup. Moreover, there exists $\gamma > 0$ and $R > 0$ such that for any $\phi \in C_R^\infty(\U(\A)M\mo \G(\A))$ and any collection $\{\lambda_x\}$ where $\lambda_x \in \mathfrak{D}_x(\gamma)$ with $\|\lambda_x \| \leq R$ and where $x$ runs over the finite set of $M$-cuspidal orbits in $N_G(M) \cap X / M$:
\[
\int_{[H]} \theta_\phi (h) dh = \sum_x \int_{\lambda_x + i(\fraka_M^*)_x^-} J(\phi[\lambda],x,\lambda) d\lambda.
\]
\end{thm}
\begin{proof}
 The convergence of the period integral of the pseudo-Eisenstein series 
 \[
 \int_{[H]} \theta_\phi (h) dh
 \]
 for large enough $R$ follows from Lemma (\ref{lem2.3}) directly.

For any $x \in N_G(M) \cap X$ and $\phi$ we define
$$I_x(\phi) = \int_{P_x \mo G_x(\A)} \phi(h\eta_x) dh.$$ Note that the convergence of the integral is a result given in \cite{AGR}.
By unfolding the integral and reinpreting the orbits where the sum runs over, we have
\begin{equation}\label{steps}
\int_{[H]} \theta_\phi (h) dh = \sum_{x \in N_G(M) \cap X/M} I_x(\phi).
\end{equation}

Indeed, unfolding the integral, we have
\begin{align*}
\int_{H \mo H(\A)} \sum_{\gamma \in P \mo G} \phi(\gamma h) dh &= \sum_{\gamma \in P \mo G} \int_{H \mo H(\A)} \phi(\gamma h) dh\\
&= \sum_{\eta \in P \mo G /H} \sum_{\delta \in H \cap \eta^{-1}P\eta \mo H} \int_{H \mo H(\A)} \phi(\eta \delta h) dh\\
&= \sum_{\eta \in P \mo G /H}  \int_{H \cap \eta^{-1}P\eta \mo H(\A)} \phi(\eta h) dh\\
&= \sum_{x \in P\mo X} \int_{P_x \mo G_x(\A)} \phi(h \eta_x) dh
\end{align*}
where $\eta_x \in G$ is such that $x = \eta \bar{\eta}^{-1}$.
Further, note that
\begin{align}\label{eq}
I_x(\phi) = \int_{P_x(\A) \mo G_x(\A)} \int_{P_x \mo P_x(\A)} \phi(ph\eta_x) \delta_{P_x}(p)^{-1} dp dh.
\end{align}
by Lemma \ref{lem3.8} (3) and the cuspidality condition on $\phi$, $I_x(\phi) = 0$ unless $x$ is $M$-admissible. By Lemma \ref{lem3.10}, the sum runs over $N_G(M) \cap X / M$.

For any $x \in N_G(M) \cap X$ which is not $M$-cuspidal, we have
\begin{align*}
I_x(\phi) = 0.
\end{align*}
Indeed, by Lemma \ref{lem3.9} and the $\U(\A)$-invariance of $\phi$, the inner integral in (\ref{eq}) transforms to
\begin{align*}
&\int_{P_x \mo \P_x(\A)} \phi(ph\eta_x) \delta_{P_x}(p)^{-1} dp \\
=& \int_{M_x \mo \M_x(\A)} \int_{U_x \mo \U_x(\A)} \phi(umh\eta_x) \delta_{P_x}(um)^{-1} du dm\\
=& \int_{M_x \mo \M_x(\A)} \phi(mh\eta_x) \delta_{P_x}(m)^{-1} dm,
\end{align*}
where we normalize the measure of $U_x\mo \U_x(\A)$ to be $1$. Note that $\phi(\cdot g)$ is cuspidal on $M \mo \M(\A)$ and the character $\delta_{P_x}$ is trivial on $\M_x(\A)^1$. By Corollary \ref{cor4.5}
\begin{align*}
&\int_{M_x \mo \M_x(\A)} \phi(mh\eta_x) \delta_{P_x}(m)^{-1} dm \\
=& 
\int_{A_{M_x}}\int_{M_x \mo \M_x(\A)^1} \phi(mah\eta_x) dm \delta_{P_x}(a)^{-1} da\\
=& 0
\end{align*}
for all $x \in N_G(M) \cap X$ which is not $M$-cuspidal. In particular,
\[
\int_{[\H]} \theta_\phi(h) dh = 0
\]
when $\M$ is not contained in the Siegel parabolic subgroup. Assume therefore that $M$ is contained in the Siegel parabolic subgroup and $x$ is $M$-cuspidal.
By Lemma \ref{lem3.9} and Lemma \ref{lem3.19}, we have
\begin{align*}
I_x(\phi) &= \int_{A_M^{M_x}U_x(\A)M_x\mo G_x(\A)} \int_{A_M^{M_x}} \phi(ah\eta_x) \delta_{P_x}^{-1}(a) da dh\\
&= \int_{A_M^{M_x}U_x(\A)M_x\mo G_x(\A)} \int_{(\fraka_M)_x^+} \phi(e^\nu h\eta_x) e^{-\langle \rho_P + \rho_x, \nu \rangle} d\nu dh.
\end{align*}

Now assume further that $\phi$ smooth, i.e. $\phi \in C_R^\infty(\U(\A)M \mo \G(\A))$. The inner integral of the above is a Fourier transform $\mathcal{F}^+$ on $(\fraka_M)_x^+$, which is also $(\mathcal{F}^-)^{-1} \circ \mathcal{F}$, where $\mathcal{F}$ is the Fourier transform on $\fraka_M$ and $\mathcal{F}^-$ is the Fourier transform on $\fraka_M^-$.Hence
\begin{align}\label{changeorder}
I_x(\phi) = \int_{A_M^{M_x}U_x(\A)M_x\mo G_x(\A)} \int_{\lambda_x + i(\fraka_M^*)_x^-} \phi[\lambda]_\lambda(h\eta_x) d\lambda dh
\end{align}
for any $\lambda_x \in \rho_x + (\fraka_{M,\mathbb{C}}^*)_x^-$ such that $\|\lambda\| < R$.

By Theorem \ref{convthm}, the double integral converges for $\mathrm{Re} \lambda_x \in \mathfrak{D}_x(\gamma)$ and suitable $R$ and $\gamma$. By Changing the order of the integrals in (\ref{changeorder}) we obtain the theorem.
\end{proof}


\section{The Spectrum}
\subsection{The $H$-distinguish Spectrum}
We will define a subspace of $L^2([\G])$, denoted $L^2_{\hdist}([\G])$, which measures the part of the spectrum which is distinguished with respect to $H$. First, let $L^2_{\hconv}([\G])$ be the subspace of $L^2([\G])$ consisting of $\varphi$ such that the integral $\int_{[\H]_G}|f * \varphi(h)| dh$ converges for any $f \in C_c(\G(\A))$(where the latter denotes the space of continuous, compactly supported functions on $\G(\A)$). Note that, except for trivial cases, $[\H]_G$ is of measure zero in $[\G]$--hence the need for convolution.
The space $L^2([\G])_{\hconv}$ contains the space of rapidly decreasing functions on $[\G]$ (cf. [\cite{AGR}, Proposition 1]). (If $\H \cap \G^{\mathrm{der}}$ is semisimple, then $L^2([\G])_{\hconv}$ contains in fact the space of bounded measurable functions on $[\G]$.) In particular, $L^2_{\hconv}([\G])$ is dense in $L^2([\G])$. Let
\[
L^2_{\hconv}([\G])^\circ = \{\varphi \in L^2_{\hconv}([\G]) \mid \int_{[H]_G} (f*\varphi)(h) dh = 0 \text{ for all } f \in C_c(\G(\A))\}.
\]
We can define the “strong” $H$-distinguished spectrum $L^2_{\hdist}([\G])^{st}$ to be the orthogonal complement in $L^2([\G])$ of $L^2([\G])_{\hconv}^\circ$. 

More generally, for any subspace $\mathcal{C}$ of $L^2([\G])_{\hconv}$ define
\[
\mathcal{C}_H^\circ = \{\phi \in \mathcal{C} \mid \int_{[\H]_G} |(f*\phi)(h)| dh = 0 \text{ for all } f \in C_c(\G(\A)\}.
\]

Note that if, for any $\phi \in \mathcal{C}$, $\phi$ is continuous and the integral $\int_{[\H]_G}|\phi(hg)|dh$ converges for all $g \in \G(\A)$, then
\[
\mathcal{C}_H^\circ = \{\phi \in \mathcal{C} \mid \phi(hg)=0 \text{ for all } g\in \G(\A)\}.
\]
To define $L^2([\G])_{H-\text{dist}}$ we will take $\mathcal{C}$ to be the space of pseudo-Eisenstein series. To make this more precise we recall some standard facts and terminology from \cite{MW}.

\subsection{Coarse Decomposition}
Let $\operatorname{\Pi_{cusp}}(A_G\backslash\G(\A))$ be the set of equivalent classes of irreducible cuspidal representation of $\G(\A)$ with trivial central character on $A_G$. A cuspidal datum of $G$ is a pair $(M,\pi)$, where $M$ is a Levi subgroup of $G$ and $\pi\in\picusp(A_M\backslash\M(\A))$. We say two cuspidal data $(M,\pi)$, $(M',\pi')$ are equivalent if there exist $\gamma\in G$ such that $\gamma M \gamma^{-1}=M'$ and $\gamma \pi=\pi'$. And we denote $(M,\pi)\sim(M',\pi')$ if $\gamma(M,\pi)=(M',\pi')$. By Bruhat decomposition, if such $\gamma$ exist, we can suppose that $\gamma$ is an element of the Weyl group $W$, and we define 
$$W((M,\pi),(M',\pi')):=\{w\in W:w(M,\pi)=(M',\pi')\}.$$

For any cuspidal data $(M,\pi)$ let $L^2_{\text{cusp},\pi}([\M])$ be the $\pi$-isotypic component of the cuspidal spectrum $L^2_{\text{cusp}}([\M])$, and let
\begin{align*}
L^2_{\text{cusp},\pi}&(\U(\A)M\mo \G(\A))\\
= &\{\varphi : \U(\A)M \mo \G(\A) \rightarrow \C \text{ measurable} \mid \\
&\delta_P^{-\frac{1}{2}}([\M])
\text{ for all }g \in \G(\A), \int_{A_M\U(\A)M \mo \G(\A)}|\varphi(g)|^2 dg < \infty\}.
\end{align*}
For any finite set of $K$-types $\mathfrak{F}$, the space $L^2_{\text{cusp},\pi}(\U(\A) M \mo \G(\A))^\mathfrak{F}$ (the direct sum of $K$-isotypic components pertaining to $\mathfrak{F}$) is finite-dimensional and consists of smooth functions.

Let $\mathfrak{X} \in \mathfrak{E}$ and $\frakF$ be a finite set of $K$-types. Given $R\gg 1$ and $(M,\pi)\in \frakX$, recall the function space $P^{R,\frakF}_{(M,\pi)}$ defined in \cite[II.1]{MW}:
\begin{align*}
    P^{R,\frakF}_{(M,\pi)}=P^R((\fraka^G_M)^*;L^2_{\cusp,\pi}(\UAMG)^\frakF)\cong P^R((\fraka^G_M)^*)\otimes L^2_{\cusp,\pi}(\UAMG)^\frakF).
\end{align*}
We denote the value of $\phi\in\PRFMpi$ at $\lambda$ by $\phi[\lambda]$. This notation is consistent with the one in section 2.6 where we view $\phi$ as a function in the space $C^\infty_R (\UAMG).$ 
Recalled the pseudo-Eisenstein series $\theta_\phi$ (see \cite[II.1.10]{MW}) is defined for all $\phi\in \PRFMpi$. Let 
$$\PRFX:=\bigoplus_{(M,\pi)\in\frakX}\PRFMpi.$$
We can extend $\theta_\phi$ to the map:
\begin{align*}
    \theta^\frakX: \PRFX&\to L^2\BG\\
    \sum_{i=1}^n c_i\phi_i &\mapsto \sum_{i=1}^n c_i\theta_{\phi_i}.
\end{align*}

Let $\frakP_{\frakX,\frakF}\BG$ be the image of $\theta^\frakX$, and let $L^2_\frakX \BG_\frakF$ be the closure of $\frakP_{\frakX,\frakF}$ in $L^2\BG$. Let $\frakP_\frakX\BG:=\bigcup_{\frakF\in \hat{K}\; \text{finite}}\frakP_{\frakX,\frakF}\BG$, and let $L^2_\frakX \BG$ be the closure of $\frakP_{\frakX}$ in $L^2\BG$. By \cite[Theorem II.2.1]{MW}, $L^2_\frakX \BG$ is independent of choice of $R$ for $R>0$.

We have the coarse decomposition of $L^2\BG$ (\cite[II.2.4]{MW})

\begin{align}
    L^2\BG=\hat{\bigoplus_{\frakX\in\frakE}}L^2_\frakX\BG.
\end{align}
In particular, the space of all pseudo-Eisenstein series  
\begin{align*}
    \frakP\BG=\bigoplus_{\frakX\in\frakE} \frakP_\frakX\BG
\end{align*}
is dense in $L^2\BG$.

As commented at the end of section 5.1, we define $L^2_{\hdist}\BG$ to be the orthogonal complement of $\frakP_\frakX\BG_H^\circ$ in $L^2\BG$. we also define 
\begin{align*}
    L^2_{\disc, \hdist}\BG:=L^2_{\hdist}\BG\cap L^2_{\disc}\BG.
\end{align*}

\subsection{Finer Decomposition}\label{finer decomp}
We will follow \cite[6.3]{LO}, which summarizes \cite[Chapter V]{MW}. Fix $\frakX\in \frakE$ and $\frakF$ as in section 5.2. We define a root hyperplane in $(\fraka_M^G)^*$ as an affine hyperplane given by equation $\langle \lambda,
\alpha^\vee\rangle=c$ for some coroot $\alpha^\vee$ correspond to root $\alpha\in\Delta_P$ and $c\in \R$. Consider certain finite set $S^\frakF_\frakX$ consists of tuples $(M,\pi, \frakS)$ where $(M,\pi)\in \frakX$ and $\frakS$ is an affine subspace of $(\fraka_M^G)^*$ which is an intersection of root hyperplanes. (As remarked in \cite{LO}, $S_\frakX^\frakF$ is merely locally finite. However, we can replace it by the finite set $\operatorname{Sing}^{G,\frakF}$ in \cite[V.3.13]{MW}.) 
We define
$$S_\frakX:=\bigcup_{\substack{\frakF\in \hat{K}\\ \frakF \text{ finite}}}S_\frakX^\frakF.$$
$S_\frakX$ consists of triples $(M,\pi,\frakS)$ where $(M,\pi)\in \frakX$ and $\frakS$ is singular hyperplane for the intertwining operator corresponding to $(M,\pi)$.

We consider the set of equivalence classes of $S_\frakX^\frakF$ and $S_\frakX$ under the equivalence relation $(M,\pi,\frakS)\sim (M',\pi',\frakS')$ if there exist $w\in W((M,\pi),(M',\pi')$ such that $w\frakS = \frakS'$. We denote the set of equivalence classes of $S_\frakX^\frakF$ and $S_\frakX$ by $[S_\frakX^\frakF]$ and $[S_\frakX]$ respectively. With the above notation, we can state the finer decomposition (\cite[V.3.13 Corollary, V.3.14 Corollary]{MW})
$$L^2_\frakX\BG_\frakF=\bigoplus_{\frakC\in[S_\frakX^\frakF]}L^2_\frakX\BG_{\frakC,\frakF}$$
and further
\begin{align*}
    L^2_\frakX\BG=\hat{\bigoplus_{\frakC\in[S_\frakX]}} L^2_\frakX\BG_\frakC
\end{align*}
where 
\[
L^2_\frakX\BG_\frakC:=\overline{\sum_\frakF L^2_\frakX\BG_{\frakC,\frakF}}.
\]

The space $L^2_\frakX\BG_\frakC$ defined in \cite[V]{MW} as a integral of certain residual of Eisenstein series. For our purpose in this paper, we will use an alternative description. 

For $\frakC$, $\frakC'\in[S_\frakX]$, we denote $\frakC\succeq\frakC$ if for $(M,\pi,\frakS)\in\frakC$, we can choose $(M',\pi',\frakS')\in\frakC'$ such that $(M,\pi)=(M',\pi')$ and $\frakS\supseteq\frakS'$. This condition is well-defined since if $(M,\pi)=(M',\pi')$ and $\frakS\supseteq\frakS'$, then $w(M,\pi)=w(M',\pi')$ and $w\frakS\supseteq w\frakS'$ for any $w\in W$. We write $\frakC'\succ\frakC$ if $\frakC\succeq\frakC'$ but $\frakC\neq\frakC'$.

Fixing $m\gg 1$, we define
\begin{align*}
    &\Tilde{P}^{R,\frakF}_{\nsucceq\frakC}:=\{\phi=(\phi_{(M,\pi)})\in P^{R,\frakF}_\frakX: \text{ for any } \frakC'=[(M,\pi,\frakS)]\in [S_\frakX],\\
    &\frakC'\nsucceq \frakC, \phi_{(M,\pi)} \text{ together with its derivatives of order 1 to m vanishes on } \frakS'\}.
\end{align*}
Similarly, we define $\Tilde{P}^{R,\frakF}_{\nsucc\frakC}$. The space $\Tilde{P}^{R,\frakF}_{\nsucceq\frakC}$ is contained in the corresponding space $P^{R,\frakF}_{\frakC, T'}$ defined in \cite[V.3.3]{MW}, ($T'$ will be determined by $R$.) and the partial order $\succeq$ is replaced by a totally order refinement. In fact, if we replace $P^{R,\frakF}_{\frakC, T'}$ by $\Tilde{P}^{R,\frakF}_{\nsucceq\frakC}$ in \cite[V.3]{MW}, then all the statements and proofs remain valid. (And the induction is based on the codimension of $\frakS$.)

We have
\begin{equation}
L^2_\frakX\BG_{\frakC,\frakF}=\overline{\{\theta_\phi: \phi\in\Tilde{P}^{R,\frakF}_{\nsucceq\frakC}\}}\cap \{\theta_\phi: \phi\in\Tilde{P}^{R,\frakF}_{\nsucc\frakC}\}^\perp\label{L^2_X C,F description}
\end{equation}
and
\begin{equation}
\bigoplus_{\frakC'\succeq \frakC} L^2_\frakX\BG_{\frakC',\frakF}=\overline{\{\theta_\phi: \phi\in\Tilde{P}^{R,\frakF}_{\nsucceq\frakC}\}}.
\end{equation}

Let $(G,H)$ is a pair of groups which Theorem \ref{main theorem} is applicable. 
Assume for any Levi subgroup $M$ there exist a finite collection of affine hyperplane of $(\fraka_M^G)^*$, $\frakA_M^H$ with the following constrain. For each element $\frakS\in\frakA_M^H$, there exist an element $\lambda_\frakS\in\frakS$ and holomorphic functions $J_\frakS(\varphi
,-)$, where $\varphi\in \Au^{rd}_P(G)$, such that in a neighborhood of $\Re\;\lambda=\lambda_\frakS$ in the ambient space $\frakS_\C$, the integrals
\begin{align}
    \int_{[\H]_G} \theta_\phi (h) dh = \sum_{\frakS\in\frakA_M^H} \int_{\lambda\in\frakS_\C:\Re\;\lambda=\lambda_\frakS} J_\frakS(\phi[\lambda],\lambda) d\lambda
\end{align}
are absolutely convergent for any $\phi\in C_R(\UAMG)$. Here, $\frakS_\C$ is the complexification of the subspace parallel to $\frakS$.

Given $\frakX\in\frakE$ and $(M,\pi)\in\frakX$, define 
\begin{align*}
    &\frakA^H_{(M,\pi)}:=\{\frakS\in\frakA_M^H: J_\frakS(\varphi,-) \text{ does not vanish identically for some }\\
    &\varphi\in\Au^{rd}_P(G)\cap L^2_{\text{cusp},\pi}(\UAMG)\},
\end{align*}
and $$\frakA^H_\frakX:=\{(M,\pi,\frakS):(M,\pi)\in\frakX, \frakS\in\frakA^H_{(M,\pi)}\}.$$
Assume that $\frakA^H_\frakX$ is a union of equivalence classes of $S_\frakX$. (Enlarge $S_\frakX$ if necessary, but in our case, $\frakA^H_\frakX\subseteq [S_\frakX]$.) Let $[\Au_\frakX^H]$ be the set of equivalent classes of $\Au_\frakX^H$.

We separate $[S_\frakX]$ into a disjoint union of two subsets:
$$[S_\frakX]_H^\circ:=\{\frakC\in [S_\frakX]: \frakC'\nsucceq\frakC \text{ for any }\frakC'\in[\frakA^H_\frakX]\},$$
$$[S_\frakX]_{H\text{-dist}}:=\{\frakC\in [S_\frakX]: \frakC'\succeq\frakC \text{ for some }\frakC'\in[\frakA^H_\frakX]\}.$$

Note that if $\frakC\in[S_\frakX]_H^\circ$, then $ \int_{[\H]_G} \theta_\phi (h) dh$ will vanish for any $\phi\in \Tilde{P}^{R,\frakF}_{\nsucceq\frakC}$.

We have the following corollary from (\ref{L^2_X C,F description}).

\begin{cor}\label{L_2 with vanishing hyperplane is in vanishing spectrum}
    For any $\frakX\in\frakC$, we have 
    \[
    \bigoplus_{\frakC\in[S_\frakX]^\circ_H}L^2_\frakX \BG_\frakC\subseteq\overline{\frakP_\frakX\BG_H^\circ}.
    \]
    Thus, 
    \[L^2_{H\text{-dist}}\BG\subseteq\hat{\bigoplus_{\frakX\in\frakC, \frakC\in\SHdist}}L^2_\frakX\BG_\frakC.\]
    
\end{cor}

\subsection{The Main Result}
Let
$$\frakB:=\{(\frakX,\frakC):\frakX\in\frakE,\frakC\in[S_\frakX]\}.$$
Recall the finer decomposition in Section 5.3
$$L^2\BG=\hat{\bigoplus_{(\frakX,\frakC)\in\frakB}}L^2_\frakX\BG_\frakC.$$
Define the following subsets of $\frakB$:
\begin{align*}
    \frakQ&=\{(\frakX,\frakC)\in\frakB:\dim \frakC=0\};\\
    \frakB_{\hdist}&=\{(\frakX,\frakC)\in\frakB: \frakC\in\SHdist \};\\
    \frakQ_{\hdist}&=\frakQ\cap\frakB_{\hdist};\\
    \Tilde{\frakB}&=\{(\frakX,\frakC)\in\frakB: \frakX\in\frakE\};\\
    \Tilde{\frakQ}&=\frakQ\cap\Tilde{\frakB};
\end{align*}
where $\frakE$ is the set of cuspidal data $(M,\pi)$, where $M$ is contained in the Siegel Levi subgroup. By theorem $\ref{main theorem}$, we have $\frakB_{\hdist}\subseteq\Tilde{\frakB}$ and $\frakQ_{\hdist}\subseteq\Tilde{\frakQ}$.

Denote 
\begin{equation*}
\tilde{L}^2\BG= \hat{\bigoplus}_{\frakX\in \tilde{\frakE}} L^2_\frakX\BG
\end{equation*}
and
\begin{equation*}
\Tilde{L}^2\BG_\disc=L^2_\disc\BG\cap \tilde L^2\BG.
\end{equation*}
Since $\frakB_{\hdist}\subseteq\Tilde{\frakB}$, by Corollary \ref{L_2 with vanishing hyperplane is in vanishing spectrum}, we have 
\begin{align}
    L^2_{\hdist}\BG\subseteq \Hat{\bigoplus_{(\frakX,\frakC)\in \frakB_{\hdist}}}L^2_\frakX\BG \subseteq \tilde{L}^2\BG.
\end{align}

\bibliographystyle{plain} 
\bibliography{refs} 
\end{document}